\documentclass[11pt, reqno]{amsart}
\usepackage[utf8]{inputenc}
\usepackage[T1]{fontenc}
\usepackage{microtype}
\usepackage{amsmath}
\usepackage{amsfonts}
\usepackage[english]{babel}
\usepackage{mathrsfs}
\usepackage{amssymb}
\usepackage{amsthm}
\usepackage{mathtools}
\usepackage[poly,arrow,curve,matrix]{xy}
\usepackage[a4paper]{geometry}
\usepackage{verbatim}
\usepackage{stmaryrd}
\usepackage{fullpage}
\usepackage{amscd}

\usepackage[backend=biber, style=numeric, maxnames=10]{biblatex}  
\usepackage{geometry}

\newcommand\blfootnote[1]{%
  \begingroup
  \renewcommand\thefootnote{}\footnote{#1}%
  \addtocounter{footnote}{-1}%
  \endgroup
} 

\DeclareFontEncoding{OT2}{}{} 

\usepackage[colorlinks]{hyperref}

\usepackage{enumerate}
\usepackage[center]{caption}

\usepackage{tikz-cd}

\def\nor{\mathop{\mbox{\normalfont N}}\nolimits}

\def\restrict#1{\raise-.5ex\hbox{\ensuremath|}_{#1}}

\def\XXint#1#2#3{{\setbox0=\hbox{$#1{#2#3}{\int}$ }
\vcenter{\hbox{$#2#3$ }}\kern-.5775\wd0}}

\def\red{\mathop{\mbox{\normalfont red}}\nolimits}

\newtheorem{theorem}{Theorem}[section]
\newtheorem{lemma}[theorem]{Lemma}
\newtheorem{proposition}[theorem]{Proposition}
\newtheorem{corollary}[theorem]{Corollary}
\newtheorem{remark}[theorem]{Remark}
\newtheorem{definition}[theorem]{Definition}

\renewenvironment{proof}[1][Proof.]{\begin{trivlist}
\item[\hskip \labelsep {\itshape #1}]}{\end{trivlist}}

\title{First-order definability of affine Campana points in the projective line over a number field}
\author{Juan Pablo De Rasis \\ \MakeLowercase{\texttt{derasis.1@osu.edu}}}

\begin{document}

\maketitle

\begin{abstract}We offer a $\forall\exists$-definition for (affine) Campana points over $\mathbb{P}^1_K$ (where $K$ is a number field), which constitute a set-theoretical filtration between $K$ and $\mathcal{O}_{K,S}$ ($S$-integers), which are well-known to be universally defined (\cite{MR3432581}, \cite{MR3207365}, \cite{MR3882159}). We also show that our formulas are uniform with respect to all possible $S$, are parameter-free as such, and we count the number of involved quantifiers and offer a bound for the degree of the defining polynomials.

\end{abstract}

\section{Introduction}

Hilbert's\blfootnote{The present research has been funded by the NSF grants DMS-1902199 and DMS-2152182 (with Dr. Jennifer Park as the principal investigator), and DMS-1748837 (with Dr. Eric Katz as the principal investigator).} Tenth problem asks for the existence of an algorithm (Turing machine) to decide whether a given polynomial with integer coefficients in any (finite) number of variables has a solution in the integers. In $1970$ Matiyasevich was able to give a negative answer to this question in \cite{MR0258744}, by completing a work previously made by Davis, Putnam, and Robinson (see \cite{MR0133227}). Thus, the existential theory over $\mathbb{Z}$ is undecidable. The analogous statement over the ring of integers of many number fields remains unknown (see \cite{MR0360513}, \cite{MR4126887}, \cite{MR4633727} for examples where an answer could be given). If $\mathbb{Z}$ were existentially defined (i.e. diophantine) over $\mathbb{Q}$, then by interpreting the existential theory of $\mathbb{Z}$ in the existential theory of $\mathbb{Q}$, undecidability of the existential theory of $\mathbb{Z}$ would yield undecidability of the existential theory of $\mathbb{Q}$ (see \cite[Proposition 2.1]{garciafritz2023effectivity}).

In $1949$ Robinson showed in \cite{MR0031446} that there is a first-order $\forall\exists\forall$-definition\footnote{In other words, there exist $m,n,k\in\mathbb{Z}_{\geq 1}$ and \mbox{$f\in \mathbb{Z}\left[T,X_1,\cdots,X_n,Y_1,\cdots,Y_m,Z_1,\cdots,Z_k\right]$} such that, given $t\in\mathbb{Q}$, we have $t\in\mathbb{Z}$ if and only if for all $x_1,\cdots,x_n\in\mathbb{Q}$ there exist $y_1,\cdots,y_m\in\mathbb{Q}$ such that for all $z_1,\cdots,z_k\in\mathbb{Q}$ we have $f\left(t,x_1,\cdots,x_n,y_1,\cdots,y_m,z_1,\cdots,z_k\right)=0$. We call this a ``$\forall\exists\forall$-definition'' because of the way in which quantifiers are introduced, and it can be codified by the formula in the free variable $t$\[\varphi\left(t\right)\coloneqq\forall x_1\cdots\forall x_n\exists y_1\cdots\exists y_m\forall z_1\cdots\forall z_k\left(f\left(t,x_1,\cdots,x_n,y_1,\cdots,y_m,z_1,\cdots,z_k\right)\neq 0\right).\]} of $\mathbb{Z}$ in $\mathbb{Q}$. In $2009$ Poonen improved it to a $\forall\exists$-definition (see \cite{MR2530851}), and in $2010$ Koenigsmann achieved a purely universal formula in \cite{MR3432581}. This was generalized to integers in number fields by Park in $2012$ (see \cite{MR3207365}) and later to $S$-integers in global fields by Eisentraeger and Morrison in \cite{MR3882159}. More recently, Daans proved in \cite{MR4378716} that if $K$ is a global field and $R$ is a finitely generated subring of $K$ whose quotient field is $K$, then $K\setminus R$ is diophantine in $K$. Moreover, in $2021$ Sun and Zhang showed that $\mathbb{Z}$ is universally defined in $\mathbb{Q}$ with $32$ universal quantifiers (that is, $32$ bounded variables) and a defining polynomial with degree at most $6\cdot 10^{11}$. Daans improved this to show in \cite{daans2023universally} that the ring of $S$-integers on a number field $K$ is universally defined with $10$ universal quantifiers.

These definitions involve first-order formulas \emph{with parameters}, which means that the defining polynomials in the formulas are allowed to have coefficients in the relevant field. In \cite{MR3343541}, Anscombe and Koenigsmann prove that, given $q\in\mathbb{Z}_{\geq 1}$ a prime power, $\mathbb{F}_q\left[\left[t\right]\right]$ is diophantine in $\mathbb{F}_q\left(\left(t\right)\right)$ with a defining polynomial with no parameters (in other words, with integer coefficients). This will motivate several observations about our uniform definitions, which will also involve polynomials with integer coefficients.

While deciding whether $\mathbb{Z}$ is diophantine in $\mathbb{Q}$ seems to be out of reach, another set of ideas try to consider Hilbert's Tenth Problem over $R$, where $R\supseteq \mathbb{Z}$ is a subring of $\mathbb{Q}$. In $2003$ Poonen showed (see \cite{MR1992832}) that Hilbert's Tenth Problem is undecidable over $R=\mathbb{Z}\left[\mathscr{S}^{-1}\right]$, where $\mathscr{S}$ is a set of primes of natural density $1$. See \cite{MR2549950}, \cite{MR2820576}, and \cite{MR2915472} for several generalizations and additions to this result and ideas. In \cite{MR3695862} Eisenträger, Miller, Park, and Shlapentokh showed that, for a set of primes $\mathscr{S}$ of lower density $0$, Hilbert's Tenth Problem over $\mathbb{Q}$ is Turing equivalent to Hilbert's Tenth Problem over $\mathbb{Z}\left[\mathscr{S}^{-1}\right]$. All these results suggest that intermediate rings (or, more generally, sets) could be interesting objects to study, and this will be the approach we will take here: we will offer a first-order description of a particular set-theoretical filtration between $\mathbb{Q}$ and $\mathbb{Z}$. This filtration will be induced by \emph{Campana points}, which were introduced to geometrically generalize $n$-full integers. They are defined in general on a smooth proper variety $X$ with respect to $\mathbb{Q}$-divisors of $X$. In the simplest case of $X=\mathbb{P}^1_\mathbb{Q}$ and the $\mathbb{Q}$-divisor $\left(1-\frac{1}{n}\right)\left\{x_1=0\right\}$ (some $n\in\mathbb{Z}_{\geq 1}$), the Campana points (in the affine line) are $C_n\coloneqq \left\{0\right\}\cup \left\{r\in\mathbb{Q}^\times:\nu_p\left(r\right)\in \mathbb{Z}_{\geq 0}\cup \mathbb{Z}_{\leq -n}\text{ for all primes $p$}\right\}$. This gives a set-theoretical filtration $\mathbb{Q}=C_1\supseteq C_2\supseteq \cdots $ whose limit is $\mathbb{Z}$, which in fact corresponds to Campana points with respect to the $\mathbb{Q}$-divisor $\left\{x_1=0\right\}$. All this can be generalized to arbitrary number fields and their rings of $S$-integers, and in this paper we prove:

\begin{theorem}\label{primobj}If $K$ is a number field, $n\in\mathbb{Z}_{\geq 2}$, and $S$ is any finite set of places containing the archimedean ones, then the set of Campana points\[C_{K,S,n}\coloneqq \left\{0\right\}\cup \left\{r\in K^\times:\nu_{\mathfrak{p}}\left(r\right)\in \mathbb{Z}_{\geq 0}\cup \mathbb{Z}_{\leq -n}\text{ for all $\mathfrak{p}$ outside $S$}\right\}\]is $\forall\exists$-definable in $K$. Moreover, the defining formula is uniform with respect to all possible $S$, and involves $838$ universal quantifiers, $558$ existential quantifiers, and a defining polynomial of degree at most $\max\left\{194n+2611,3387\right\}$. This degree is at most $\max\left\{2n+19,27\right\}$ if $K\subseteq \mathbb{R}$.

\end{theorem}

Observe that every recursively enumerable subset of $\mathbb{Q}$ is $\exists\forall$-definable, with $10$ existential quantifiers and $9$ universal quantifiers (see \cite[Corollary 6.2]{daans2023universally}), and this extends to an analogous statement to number fields over whose rings of integers we have an existential definition for $\mathbb{Z}$. We know of cases where this is the case (see \cite{MR0360513}, \cite{MR4126887}, \cite{MR4633727}), and it is conjectured to hold for any number field. The advantage of our approach, besides the fact that it already holds for any number field, relies mainly on having a clear control on the uniformity while having at the same time an explicit understanding of the complexity of the involved formulas.

Throughout all this paper, for each number field $K$ we let $\mathcal{O}_K$ be its ring of algebraic integers. We let $\Omega_K^\infty$ be the set of its infinite places and $\Omega_K^{<\infty}$ be the set of all its finite places. We let $\Omega_K\coloneqq \Omega_K^\infty\cup \Omega_K^{<\infty}$ be the set of all places of $K$. For each place $v$ of $K$ (finite or infinite) we let $K_v$ be the completion of $K$ with respect to $v$ and we let $\mathbb{F}_v$ be its residue field when $v\in \Omega_K^{<\infty}$. For each $\mathfrak{p}\in \Omega_K^{<\infty}$ we let $\mathcal{O}_{K,\mathfrak{p}}$ be the ring of integers of $K_\mathfrak{p}$, so that $K\cap \mathcal{O}_{K,\mathfrak{p}}=\left(\mathcal{O}_K\right)_{\mathfrak{p}}$ is the set of elements of $K$ which are integral with respect to $\mathfrak{p}$. Finally, if $S$ is a finite subset of $\Omega_K$ containing $\Omega_K^\infty$, we let $\displaystyle{\mathcal{O}_{K,S}\coloneqq \bigcap_{\mathfrak{p}\in \Omega_K\setminus S}\left(\mathcal{O}_K\right)_\mathfrak{p}}$ be the ring of $S$-integers of $K$: the set of elements of $\mathcal{O}_K$ which are integral outside $S$.

The method I will be introducing can be summarized by the following statement:

\begin{theorem}\label{derasis}Let $K$ be any number field. For each $\left(a,b,c,d\right)\in \left(K^\times\right)^4$ there exists a finite subset $\Omega_{a,b,c,d,K}$ of $\Omega_K^{<\infty}$ such that:

\begin{enumerate}[(1)]

\item For any finite subset $S$ of $\Omega_K^{<\infty}$, there exists $\left(a,b,c,d\right)\in \left(K^\times\right)^4$ such that $S=\Omega_{a,b,c,d,K}$.

\item The sets\[\left\{\left(a,b,c,d,r\right)\in \left(K^\times\right)^4\times K:r\in\bigcap_{\mathfrak{p}\in\Omega_{a,b,c,d,K}}\mathfrak{p}\left(\mathcal{O}_K\right)_{\mathfrak{p}}\right\},\]\[\left\{\left(a,b,c,d,a',b',c',d'\right)\in \left(K^\times\right)^8:\Omega_{a,b,c,d,K}\cap\Omega_{a',b',c',d',K}=\emptyset\right\}\]are diophantine over $K$

\end{enumerate}
\end{theorem}

Theorem \ref{derasis} follows from Corollary \ref{jabjab}, Theorem \ref{triunfo}, and Proposition \ref{intvacia}, proved below. It will be a very powerful tool to describe subsets of number fields defined in terms of valuations over specific places. A more robust version of this method can also be used to describe Darmon points (see \cite[§2.4]{mitankin2023semiintegral}), which will be the subject of a future paper.

We will begin with Section \ref{introductiontocampanapoints} by defining $S$-Campana points over smooth proper varieties $X$ over a number field $K$ (where $S$ is a finite set of places of $K$ containing the archimedean ones) with respect to a $\mathbb{Q}$-divisor $D$ of $X$ in which all the coefficients have the form $1-\frac{1}{n}$ for $n\in\mathbb{Z}_{\geq 1}\cup\left\{+\infty\right\}$, and having strict normal crossings on $X$. We will then give an explicit description of Campana points in the case in which $X=\mathbb{P}^1_K$ and $D=\left(1-\frac{1}{n}\right)\left\{x_1=0\right\}$ for some $n\in\mathbb{Z}_{\geq 1}$, getting the set described in Theorem \ref{primobj}. We will then do the same when $n=+\infty$ and obtain the $S$-integers.

In Section \ref{preliminares} we will begin the objective of giving a first-order definition of these Campana points. In Sections \ref{quatalg} and \ref{hsym} we state the basic results on Quaternion Algebras and Hilbert Symbols we will make use of, and in Section \ref{imppdiof} we will apply them to construct the basic diophantine sets we will be combining to end up constructing Campana points.

In Section \ref{pruebadelteom} we prove Theorem \ref{derasis} and in Section \ref{primerdescripcion} we put all this together to build a first-order definition for Campana points. We conclude with Section \ref{explicito} by analyzing the explicit first-order formulas to count the number of quantifiers and bound the degrees of the defining polynomials, and Section \ref{mainqwe} in which we derive the analogous results for more general Campana points.

\section*{Acknowledgements}

I thank my supervisor Jennifer Park, principal investigator of two of the three grants that supported this research, and whose suggestions, proof-readings and comments were extremely valuable during the whole process of this paper. Thanks also to Eric Katz, the principal investigator of the grant under which I was funded during the last stage of this research.

I would also like to thank the Mathematical Sciences Research Institute for the invitation to present the first version of all these results in the program \emph{Definability, Decidability, and Computability in Number Theory, part 2} (UC Berkeley). Also the organizers of \emph{DDC@Cambridge: FRG Workshop on Definability, Decidability, and Computability} (Harvard University) for the possibility to present the advanced version.

Thanks to Gabe Conant, Nicolas Daans, Philip Dittmann, Arno Fehm, Hunter Handley, Hector Pastén, Stefan Patrikis, Marta Pieropan, Gabriela Pinto, Jakob Stix, and Caroline Terry, for the conversations they held with me, their suggestions on further reading and research, and their extremely useful and valuable insight and advices, comments, suggestions, and questions.

\section{Campana points}\label{introductiontocampanapoints}

\subsection{Definition of Campana Points}

We begin by defining the notion of \emph{Campana orbifold}, from which we will derive Campana points. See \cite[Section §3]{MR4307130} for more details.

\begin{definition}\label{camporb}Let $X$ be a smooth variety over a field $K$, and fix an effective Weil $\mathbb{Q}$-divisor $D$ on $X$. Moreover, assume $D=\displaystyle{\sum_{\alpha\in\mathcal{A}}\varepsilon_\alpha D_\alpha}$, where $\mathcal{A}$ is a finite set and, for each $\alpha\in\mathcal{A}$, $D_\alpha$ is a prime divisor, and $\varepsilon_\alpha\in \mathfrak{W}\coloneqq \left\{1-\frac{1}{n}:n\in\mathbb{Z}_{\geq 1}\right\}\cup\left\{1\right\}$. Also assume $D_{\red}\coloneqq\displaystyle{\sum_{\alpha\in\mathcal{A}}D_\alpha}$ is a divisor with strict normal crossings on $X$. We then say that the pair $\left(X,D\right)$ is a \emph{Campana orbifold}.

\end{definition}

Fix $K$, $X$, $D$, and $\mathcal{A}$ as in Definition \ref{camporb}. Moreover, assume further that $K$ is a number field and that $X$ is proper over $K$. Fix a finite subset $S$ of $\Omega_K$ containing $\Omega_K^\infty$. A \emph{model} of $\left(X,D\right)$ over $\mathcal{O}_{K,S}$ is, following \cite{MR4405664}, a pair $\left(\mathcal{X},\mathcal{D}\right)$, where $\mathcal{X}$ is a flat proper scheme over $\mathcal{O}_{K,S}$ with $\mathcal{X}_{\left(0\right)}\cong X$ (where $\mathcal{X}_{\left(0\right)}$ is the generic fiber of the morphism $\mathcal{X}\to\operatorname{Spec}\left(\mathcal{O}_{K,S}\right)$) and  $\mathcal{D}\coloneqq \displaystyle{\sum_{\alpha\in\mathcal{A}}\varepsilon_\alpha \mathcal{D}_\alpha}$, where for each $\alpha\in\mathcal{A}$, $\mathcal{D}_\alpha$ is the Zariski closure of $D_\alpha$ in $\mathcal{X}$ (we can make sense of this by using the topological embedding $\mathcal{X}_{\left(0\right)}\hookrightarrow \mathcal{X}$ and the isomorphism $\mathcal{X}_{\left(0\right)}\cong X\supseteq D_\alpha$). We let $\mathcal{D}_{\red}\coloneqq \displaystyle{\sum_{\alpha\in\mathcal{A}}\mathcal{D}_\alpha}$.

If $P\in X\left(K\right)\coloneqq \operatorname{Hom}_{\operatorname{Sch}}\left(\operatorname{Spec}\left(K\right),X\right)$, we get a commutative diagram\begin{equation*}
\begin{tikzcd}
\operatorname{Spec}\left(K\right) \arrow[d] \arrow[r, "P"]                                                               & X\cong \mathcal{X}_{\left(0\right)} \arrow[r, hook] & \mathcal{X} \arrow[d]                               \\
{\operatorname{Spec}\left(\mathcal{O}_{K,S}\right)} \arrow[rr, "\operatorname{id}"'] \arrow[rru, "\mathcal{P}"', dashed] &                                                     & {\operatorname{Spec}\left(\mathcal{O}_{K,S}\right)}
\end{tikzcd}\end{equation*}where the map $\mathcal{P}\in \mathcal{X}\left(\mathcal{O}_{K,S}\right)$ is induced by the valuative criterion of properness (see \cite[Theorem 4.7, p. 101]{MR0463157}).\footnote{This criterion can be applied because, even though $\mathcal{O}_{K,S}$ is not a valuation ring (it is not even a local ring), localizing at every place provides us with local rings over which the criterion can be applied to get local extensions of the point that can be glued to a global extension.} In other words, each $P\in X\left(K\right)$ induces some $\mathcal{P}\in\mathcal{X}\left(\mathcal{O}_{K,S}\right)$. For each $v\in\Omega_K\setminus S$, this and the inclusion $\mathcal{O}_{K,S}\subseteq\mathcal{O}_{K,v}$ give an element $\mathcal{P}_v\in \mathcal{X}\left(\mathcal{O}_{K,v}\right)$. Let $\mathcal{A}_\varepsilon\coloneqq \left\{\alpha\in\mathcal{A}:\varepsilon_\alpha\neq 0\right\}$ and $\displaystyle{X^0\coloneqq X\setminus \bigcup_{\alpha\in\mathcal{A}_\varepsilon}D_\alpha}$. If $\alpha\in\mathcal{A}$ and $\mathcal{P}_v\not\subseteq \mathcal{D}_\alpha$ (recall that $\mathcal{P}_v\in\mathcal{X}\left(\mathcal{O}_{K,v}\right)$ has image contained in $\mathcal{X}\supseteq \mathcal{D}_\alpha$), we define $n_v\left(\mathcal{D}_\alpha,P\right)$ (the \emph{intersection multiplicity of $P$ and $\mathcal{D}_\alpha$}) to be the colength of the non-zero ideal in $\mathcal{O}_{K,v}$ defined by the pullback of $\mathcal{D}_\alpha$ via $\mathcal{P}_v$. See \cite[Definition 3.4]{balestrieri2023campana} for more details.

When $P\in D_\alpha$ we define $n_v\left(\mathcal{D}_\alpha,P\right)\coloneqq+\infty$. Intersection numbers will allow us to impose bound restrictions in order to geometrically generalize $m$-full integers, for various $m\in\mathbb{Z}_{\geq 1}$, a notion that we now introduce as Campana points. More precisely:

\begin{definition}\label{camppoints}$P\in X\left(K\right)$ is a \emph{Campana $\mathcal{O}_{K,S}$-point} if $\displaystyle{P\in \left(X\setminus\bigcup_{\substack{\alpha\in\mathcal{A}\\\varepsilon_\alpha=1}}D_\alpha\right)\left(\mathcal{O}_{K,S}\right)}$ (that is, for all $\alpha\in\mathcal{A}$ such that $\varepsilon_\alpha=1$, we have $n_v\left(\mathcal{D}_\alpha,P\right)=0$ for all $v\in\Omega_K\setminus S$), and for all $\alpha\in\mathcal{A}$ with $\varepsilon_\alpha<1$ and all $v\in\Omega_K\setminus S$ we have $n_v\left(\mathcal{D}_\alpha,P\right)^2\geq \frac{1}{1-\varepsilon_\alpha}n_v\left(\mathcal{D}_\alpha,P\right)$.

\end{definition}

\subsection{Campana points over number fields}

In this section we will carry out the above definitions to explictly characterize Campana points in projective spaces over number fields. The usual first example of Campana points is given over projective spaces over $\mathbb{Q}$ (see \cite[Section §3.2.1]{MR4307130}). We generalize that idea to arbitrary number fields. In order to do that, we need to tackle the difficulty that integers on number fields are not necessarily principal ideal domains, which makes projective points more difficult to characterize. We use the following lemma.

\begin{lemma}\label{ellem}Let $K$ be a number field, let $S$ be a finite subset of $\Omega_K$ containing $\Omega_K^{\infty}$, and let $x_0,x_1\in \mathcal{O}_{K,S}$ be such that $x_1\neq 0$. Fix $\mathfrak{p}\in\Omega_K\setminus S$ and assume $\frac{x_0}{x_1}=\frac{a}{b}\in K_\mathfrak{p}$, where $a,b\in \mathcal{O}_{K,\mathfrak{p}}$ are relatively prime in $\mathcal{O}_{K,\mathfrak{p}}$. Then the exponent of $\mathfrak{p}$ in the factorization of the fractional ideal $\left(x_1\right)\left(x_0,x_1\right)^{-1}$ is $\nu_{\mathfrak{p}}\left(b\right)$.

\end{lemma}

\begin{proof}Since $\mathcal{O}_{K,\mathfrak{p}}$ is a principal ideal domain, the ideal $\left(x_0,x_1\right)$ in $\mathcal{O}_{K,\mathfrak{p}}$ is generated by a greatest common divisor (unique up to unit scaling) $d$ of $x_0$ and $x_1$, thus the exponent of $\mathfrak{p}$ in the factorization of the fractional ideal $\left(x_1\right)\left(x_0,x_1\right)^{-1}$ equals $\nu_\mathfrak{p}\left(\frac{x_1}{d}\right)$, which equals $\nu_\mathfrak{p}\left(b\right)$ because $b$ and $\frac{x_1}{d}$ differ by a unit factor (both $\frac{a}{b}$ and $\frac{x_0/d}{x_1/d}$ are a reduced form of the fraction $\frac{x_0}{x_1}$ in the unique factorization domain $\mathcal{O}_{K,\mathfrak{p}}$). $\blacksquare$

\end{proof}

Assume $K$ is a number field, fix a finite subset $S$ of $\Omega_K$ containing $\Omega_K^\infty$ and assume $X\coloneqq \mathbb{P}^n_K$ for some $n\in\mathbb{Z}_{\geq 1}$. Fix $k\in\left[0,n\right]\cap\mathbb{Z}$ and for each $i\in\left[k,n\right]\cap\mathbb{Z}$ let $m_i\in\mathbb{Z}_{\geq 1}$, $\varepsilon_i\coloneqq 1-\frac{1}{m_i}\in\mathfrak{W}$, and $D_i\coloneqq \left\{x_i=0\right\}$, so that $\displaystyle{\left(X,\sum_{i=k}^n\varepsilon_iD_i\right)}$ is a Campana orbifold having an integral model over $\mathcal{O}_{K,S}$ given by $\displaystyle{\left(\mathcal{X},\sum_{i=k}^n\varepsilon_i\mathcal{D}_i\right)}$, where $\mathcal{X}\coloneqq \mathbb{P}^n_{\mathcal{O}_{K,S}}$.

Now fix $P\in X\left(K\right)$ and write it as $P=\left(x_0:\cdots :x_n\right)$, where $x_0,\cdots, x_n\in\mathcal{O}_{K,S}$. If $\mathfrak{p}\in \Omega_K^{<\infty}$, then since $\mathcal{O}_{K,\mathfrak{p}}$ is a principal ideal domain, we can write $P=\left(x_{0,\mathfrak{p}},\cdots,x_{n,\mathfrak{p}}\right)\in \mathbb{P}^n_{K_\mathfrak{p}}\left(K_\mathfrak{p}\right)$, where for each $i\in\left[0,n\right]\cap\mathbb{Z}$ we have $x_{i,\mathfrak{p}}\in\mathcal{O}_{K,\mathfrak{p}}$ and $\operatorname{gcd}\left(x_{0,\mathfrak{p}},\cdots,x_{n,\mathfrak{p}}\right)=1$ (equivalently, these elements generate the unit ideal in $\mathcal{O}_{K,\mathfrak{p}}$, because on principal ideal domains, maximal ideals are exactly the ones generated by a single irreducible element). Moreover, this representation is unique up to scalar multiplication by an element of $\mathcal{O}_{K,\mathfrak{p}}^\times$. For each $\mathfrak{p}\in \Omega_K\setminus S$ and $i\in\left[k,n\right]\cap\mathbb{Z}$ we have $n_\mathfrak{p}\left(\mathcal{D}_i,P\right)=\nu_\mathfrak{p}\left(x_{i,\mathfrak{p}}\right)$ (that is, the colength of the ideal of $\mathcal{O}_{K,\mathfrak{p}}$ generated by $x_{i,\mathfrak{p}}$). We conclude that $P$ is a Campana $\mathcal{O}_{K,S}$-point in $\displaystyle{\left(\mathcal{X},\sum_{i=k}^n\varepsilon_i\mathcal{D}_i\right)}$ if and only if for each $i\in\left[0,n\right]\cap\mathbb{Z}$ and $\mathfrak{p}\in\Omega_K\setminus S$ we have $\nu_{\mathfrak{p}}\left(x_{j,\mathfrak{p}}\right)^2\geq \frac{1}{1-\varepsilon_j}\nu_{\mathfrak{p}}\left(x_{j,\mathfrak{p}}\right)=m_j\nu_{\mathfrak{p}}\left(x_{j,\mathfrak{p}}\right)$ for \mbox{all $j\in\left[k,n\right]\cap\mathbb{Z}$}.

In particular, set $n=k=1$ in the above example, so that \mbox{$P=\left(x_0:x_1\right)\in\mathbb{P}^1_K\left(K\right)$} with $x_0,x_1\in\mathcal{O}_{K,S}$, and our Campana orbifold is $\left(\mathbb{P}^1_{K},\left(1-\frac{1}{m_1}\right)\left\{x_1=0\right\}\right)$. Identifying $K$ with $\left\{\left(\lambda:1\right):\lambda\in K\right\}$ and using Lemma \ref{ellem} we get that $P=\frac{x_0}{x_1}\in K$ is Campana if and only if\[\text{$\nu_{\mathfrak{p}}\left(\left(x_1\right)\left(x_0,x_1\right)^{-1}\right)^2\geq m_1\nu_{\mathfrak{p}}\left(\left(x_1\right)\left(x_0,x_1\right)^{-1}\right)$ for all $\mathfrak{p}\in \Omega_K\setminus S$}\](observe that, for any such $\mathfrak{p}$, the number $\nu_{\mathfrak{p}}\left(\left(x_1\right)\left(x_0,x_1\right)^{-1}\right)$ is independent of the choice of $S$-integral projective coordinates for $P$). To characterize these elements of $K$, we prove:

\begin{lemma}\label{ellem2}Let $K$ be a number field, let $S$ be a finite subset of $\Omega_K$ containing $\Omega_K^{\infty}$, and fix $r\in K^\times$, $\mathfrak{p}\in\Omega_K\setminus S$, and $n\in\mathbb{Z}_{\geq 1}$. Let $x_0,x_1\in \mathcal{O}_{K,S}$ be such that $r=\frac{x_0}{x_1}$. Then $\nu_\mathfrak{p}\left(r\right)\in \mathbb{Z}_{\geq 0}\cup \mathbb{Z}_{\leq -n}$ if and only if $\nu_{\mathfrak{p}}\left(\left(x_1\right)\left(x_0,x_1\right)^{-1}\right)^2\geq n\nu_{\mathfrak{p}}\left(\left(x_1\right)\left(x_0,x_1\right)^{-1}\right)$.

\end{lemma}

\begin{proof}Let $d\in\mathcal{O}_{K,S}$ be a generator of the ideal $\left(x_0,x_1\right)$ of $\mathcal{O}_{K,\mathfrak{p}}$, let $a\coloneqq \frac{x_0}{d}$ and $b\coloneqq \frac{x_1}{d}$. By Lemma \ref{ellem} we have $\nu_{\mathfrak{p}}\left(\left(x_1\right)\left(x_0,x_1\right)^{-1}\right)=\nu_{\mathfrak{p}}\left(b\right)$, so we want to show that\[\text{$\nu_\mathfrak{p}\left(r\right)\in \mathbb{Z}_{\geq 0}\cup \mathbb{Z}_{\leq -n}$ if and only if $\nu_{\mathfrak{p}}\left(b\right)^2\geq n\nu_{\mathfrak{p}}\left(b\right)$.}\]Assume $\nu_\mathfrak{p}\left(r\right)\in \mathbb{Z}_{\geq 0}\cup \mathbb{Z}_{\leq -n}$. It follows that $\nu_\mathfrak{p}\left(a\right)-\nu_\mathfrak{p}\left(b\right)\in \mathbb{Z}_{\geq 0}\cup \mathbb{Z}_{\leq -n}$, because $r=\frac{a}{b}\in K_\mathfrak{p}$. Since $a$ and $b$ are relatively prime in $\mathcal{O}_{K,\mathfrak{p}}$ we have $\nu_\mathfrak{p}\left(a\right)\nu_\mathfrak{p}\left(b\right)=0$. If $\nu_\mathfrak{p}\left(a\right)-\nu_\mathfrak{p}\left(b\right)\in \mathbb{Z}_{\geq 0}$, which is the same as $\nu_\mathfrak{p}\left(a\right)\geq \nu_\mathfrak{p}\left(b\right)$, necessarily we must have $\nu_\mathfrak{p}\left(b\right)=0$; while if $\nu_\mathfrak{p}\left(a\right)-\nu_\mathfrak{p}\left(b\right)\in \mathbb{Z}_{\geq -n}$, which is the same as $\nu_\mathfrak{p}\left(a\right)+n\leq \nu_\mathfrak{p}\left(b\right)$, we must have $\nu_\mathfrak{p}\left(b\right)\geq n$. In both cases we get $\nu_{\mathfrak{p}}\left(b\right)^2\geq n\nu_{\mathfrak{p}}\left(b\right)$, as desired.

Conversely, assume $\nu_{\mathfrak{p}}\left(b\right)^2\geq n\nu_{\mathfrak{p}}\left(b\right)$. If $\nu_{\mathfrak{p}}\left(b\right)=0$ then $r=\frac{a}{b}$ is integral with respect to $\mathfrak{p}$, giving $\nu_\mathfrak{p}\left(r\right)\geq 0$. If $\nu_{\mathfrak{p}}\left(b\right)\neq 0$ then $\nu_{\mathfrak{p}}\left(a\right)=0$ (recall that $\nu_\mathfrak{p}\left(a\right)\nu_\mathfrak{p}\left(b\right)=0$) and $\nu_{\mathfrak{p}}\left(b\right)\geq n$, hence\[\nu_\mathfrak{p}\left(r\right)=\nu_\mathfrak{p}\left(a\right)-\nu_\mathfrak{p}\left(b\right)=-\nu_\mathfrak{p}\left(b\right)\in\mathbb{Z}_{\leq -n},\]completing the proof. $\blacksquare$

\end{proof}

Therefore, given any $n\in\mathbb{Z}_{\geq 1}$, the affine $\mathcal{O}_{K,S}$-Campana points of $\mathbb{P}^1_K$ with respect to $\left(1-\frac{1}{n}\right)\left\{x_1=0\right\}$ are precisely the elements of the set\[C_{K,S,n}\coloneqq \left\{0\right\}\cup\left\{r\in K^\times:\forall \mathfrak{p}\in \Omega_K\setminus S\left(\nu_\mathfrak{p}\left(r\right)\in \mathbb{Z}_{\geq 0}\cup \mathbb{Z}_{\leq -n}\right)\right\}.\]From this we immediately obtain $C_{K,S,1}=K$, $C_{K,S,j+1}\subseteq C_{K,S,j}$ for each $j\in\mathbb{Z}_{\geq 1}$, and\[\bigcap_{j\in\mathbb{Z}_{\geq 1}}C_{K,S,j}=\bigcap_{\mathfrak{p}\in \Omega_K\setminus S}\left(\mathcal{O}_{K}\right)_\mathfrak{p}=\mathcal{O}_{K,S}.\]Thus, the natural motivation for studying Campana points for the sake of Hilbert's Tenth problem is that they provide a set-theoretical filtration between a number field and its ring of $S$-integers.

To conclude, we show that $S$-integers also arise as Campana points. Indeed, observe that if we look for Campana points with respect to the divisor $\left\{x_1=0\right\}$ in $K$, we look for fractions of the form $\frac{x_0}{x_1}$, with $x_0,x_1\in\mathcal{O}_{K,S}$ such that, for each $\mathfrak{p}\in\Omega_K\setminus S$, if $d$ is the generator of the ideal $\left(x_0,x_1\right)$ in $\mathcal{O}_{K,\mathfrak{p}}$, then $\nu_\mathfrak{p}\left(\frac{x_1}{d}\right)=0$; which is the same as saying $\nu_\mathfrak{p}\left(\frac{x_0}{x_1}\right)\geq 0$. So this is nothing but the set of $S$-integers of $K$. As mentioned in the introduction, much has been done to understand $S$-integers from a first-order point of view, so in this paper we will mainly focus on the other stages of this filtration.

\section{Quaternion Algebras, Hilbert Symbol, and Diophantine sets}
\label{preliminares}

We begin with two sections exposing the most basic results we will use from the theory of Hilbert Symbols and Central Simple Algebras, then we explain the connection between these two concepts, and proceed to build up our setup to prove our desired first-order characterizations.

\subsection{Quaternion algebras}\label{quatalg}

Given a field $k$ of characteristic different from $2$, a \emph{quaternion algebra} over $k$ is a $k$-algebra of the form $H_{a,b,k}\coloneqq \displaystyle{\bigoplus_{i,j\in \left\{0,1\right\}}k\alpha^i\beta^j}$ (where $\alpha$ and $\beta$ are formal symbols), with multiplication defined as $\alpha^2\coloneqq a$, $\beta^2\coloneqq b$, $\alpha\beta=-\beta\alpha$, where $a$ and $b$ are fixed nonzero elements of $k$. It is said to be \emph{split} if there exists a $k$-algebra isomorphism $H_{a,b,k}\cong M_2\left(k\right)$.

Over $H_{a,b,k}$, the \emph{reduced norm} of an element $x_1+x_2\alpha +x_3\beta+x_4\alpha\beta$ (where $x_1,x_2,x_3,x_4\in k$) is defined as $x_1^2-ax_2^2-bx_3^2+abx_4^2$, while the \emph{reduced trace} of such an element is defined as $2x_1$.

The following fact, proven in \cite[Proposition 1.3.2]{MR3727161}, will be the key to connect quaternion algebras with Hilbert Symbols, introduced in the next section.

\begin{lemma}\label{connect}If $k$ is a field, $\operatorname{Char}\left(k\right)\neq 2$, and $a,b\in k^\times$, then $H_{a,b,k}$ is split if and only if $ax^2+by^2=z^2$ has a nontrivial solution in $k$.\end{lemma}

\subsection{Hilbert Symbols}\label{hsym}

Given a field $k$, the (quadratic) \emph{Hilbert Symbol} is the function\[\left(-,-\right)_k:k^\times\times k^\times\to \left\{\pm 1\right\}\]defined as\[\left(a,b\right)_k\coloneqq\begin{cases}1,&z^2-ax^2-by^2=0\text{ has a nontrivial solution in $k$},\\-1,&\text{otherwise}.\end{cases}\]If $K$ is a number field and $v\in\Omega_K$, we denote by $\left(-,-\right)_v$ the quadratic Hilbert Symbol of $K_v$.

We state two facts on Hilbert Symbols. The first is a particular case of \mbox{\cite[XIV §2, Proposition 7]{MR0554237}} (see also \cite[§III, Theorem 2]{MR0344216}), and the second is \cite[Theorem 3.7]{MR3207365}.

\begin{proposition}\label{lin}If $k$ is a local field of characteristic different from $2$ and $a\in k^\times$, then $\left(a,b\right)_k=1$ for all $b\in k^\times$ if and only if $a$ is a square in $k$.

\end{proposition}

\begin{theorem}\label{epsiv2}Let $K$ be a number field, let $I$ be a finite set and fix $\left\{a_i\right\}_{i\in I}\subseteq K^\times$. For each $\left(i,v\right)\in I\times \Omega_K$ fix $\varepsilon_{i,v}\in\left\{-1,1\right\}$. The following are equivalent:

\begin{itemize}

\item There exists $x\in K^\times$ such that $\left(a_i,x\right)_v=\varepsilon_{i,v}$ for all $\left(i,v\right)\in I\times \Omega_K$.

\item All these conditions hold:

\begin{enumerate}

\item $\varepsilon_{i,v}=1$ for all but finitely many $\left(i,v\right)\in I\times \Omega_K$.

\item For all $i\in I$ we have $\displaystyle{\prod_{v\in \Omega_K}\varepsilon_{i,v}=1}$.

\item For all $v\in \Omega_K$ there exists $x_v\in K^\times$ such that $\left(a_i,x_v\right)_v=\varepsilon_{i,v}$ for all $i\in I$.

\end{enumerate}

\end{itemize}

\end{theorem}

\subsection{Relevant diophantine sets}\label{imppdiof}In this section we will formally define diophantine sets and then list and construct several of them.

\begin{definition}Given any unital commutative ring $R$ and $n\in\mathbb{Z}_{\geq 1}$, we say that a given set $A\subseteq R^n$ is \emph{diophantine over $R$}, or \emph{first-order existentially defined over $R$} (or simply "existentially defined") if there exists $m\in\mathbb{Z}_{\geq 0}$ and $P\in R\left[X_1,\cdots,X_m,Y_1,\cdots, Y_n\right]$ such that, for any $a=\left(a_1,\cdots,a_n\right)\in R^n$, we have $a\in A$ if and only if there exist $x_1,\cdots,x_m\in R$ such that $P\left(x_1,\cdots,x_m,a_1,\cdots,a_n\right)=0$. In other words, $A$ is the set-theoretical projection of the set of solutions to a given polynomial with coefficients in $R$.\footnote{In model-theoretic terms, this means that $A$ is defined by an existential first-order formula with parameters in $R$ whose quantifier-free part is a polynomial equality. We can express this by the notation $R\vDash\exists x_1\cdots \exists x_m\left(P\left(x_1,\cdots,x_m,a_1,\cdots,a_n\right)=0\right)$.} We say that $A$ is \emph{first-order universally defined over $R$} (or simply "universally defined") if $R^n\setminus A$ is diophantine over $R$.\footnote{Equivalently in model-theoretic terms, it can be defined by a universal first-order formula (with parameters in $K$) whose quantifier-free part is the negation of a polynomial equality.}
\end{definition}

When constructing diophantine or universal sets in terms of other such sets, we need to be able to reduce a system of polynomial equations to a single equation. The following fact will allow us to do this, while keeping track of the degrees of the involved polynomials.

\begin{lemma}\label{intofdiof}Let $K$ be a number field and let $f$ and $g$ be two polynomials in any (finite) number of variables. Then there exists a polynomial $F$ of degree at most $2\max\left\{\deg\left(f\right),\deg\left(g\right)\right\}$ whose zeros are exactly the common zeros of $f$ and $g$.

\end{lemma}

\begin{proof}Since $K/\mathbb{Q}$ is a finite extension, we cannot have $\mathbb{Q}\left(\sqrt{n}\right)\subseteq K$ for all $n\in\mathbb{Z}$. Fix $n_K\in\mathbb{Z}$ such that $x^2-n_K\in \mathbb{Z}\left[x\right]$ has no solutions in $K$. Observe that if $a,b\in K$ then $a^2-n_Kb^2=0$ if and only if $a=b=0$, so we can take $F\coloneqq f^2-n_Kg^2$. $\blacksquare$

\end{proof}

Lemma \ref{intofdiof} can be generalized to show that over integral domains whose quotient field is not algebraically closed, any set of finitely many polynomial equations can be reduced to a single one (see \cite[Lemma 1.2.3]{MR2297245}).

\begin{remark}\label{bounddegree}If $n\in\mathbb{Z}_{\geq 1}$ and $K$ is a number field, the intersection of $n$ diophantine subsets of $K$ is again diophantine, and its defining polynomial has degree at most $2^{n-1}d$, where $d$ is the maximum degree among the polynomials defining each of the $n$ intersected sets.

\end{remark}

Proposition \ref{intdiofopt} will allow us to significantly improve the bound in Remark \ref{bounddegree}.

We can use all these observations to build several diophantine subsets which will combine into the sets whose first-order definitions we are looking for. Most of the following notations and definitions are taken from \cite{MR3207365}, with some slight modification.

Fix a number field $K$ and $a,b\in K^\times$. The set of reduced traces of elements of $H_{a,b,K}$ having reduced norm $1$ is\[S_{a,b,K}\coloneqq \left\{2x_1:\left(x_1,x_2,x_3,x_4\right)\in K\times K\times K\times K\wedge x_1^2-ax_2^2-bx_3^2+abx_4^2=1\right\},\]which is diophantine since it is defined by the formula\[\exists x_2\exists x_3\exists x_4\left(x^2-4ax_2^2-4bx_3^2+4abx_4^2-4=0\right).\]This diophantine definition is \emph{uniform} over all such $a,b\in K^\times$, in the sense that if we want to define a specific set $S_{a,b,K}$, we only have to substitute for the corresponding $a,b\in K^\times$. Moreover, if we allow $a$ and $b$ to be variables, the resulting polynomial has integer coefficients.

This same integral-uniformity phenomenon occurs with the set $T_{a,b,K}\coloneqq S_{a,b,K}+S_{a,b,K}$, which can be defined in the free variable $x$ through the first-order formula $\exists y\left(y\in S_{a,b,K}\wedge x-y\in S_{a,b,K}\right)$, which is also a diophantine definition by Lemma \ref{intofdiof}. Continuing this reasoning we define the following sets, which are also given in terms of $a$ and $b$ (and $K$), as well as of $c,d\in K^\times$, and where the same integral-uniformity phenomenon occurs:

\begin{itemize}

\item $T_{a,b,K}^\times\coloneqq\left\{u\in T_{a,b,K}:\exists v\in T_{a,b,K}\left(uv=1\right)\right\}$, which can be defined by the existential first-order formula $u\in T_{a,b,K}\wedge \exists v\left(v\in T_{a,b,K}\wedge uv=1\right)$.

\item $I_{a,b,K}^c\coloneqq c\cdot K^2\cdot T_{a,b,K}^\times \cap \left(1-K^2\cdot T_{a,b,K}^\times\right)$, which can be defined by the existential first-order formula $\exists y_1\exists z_1\exists y_2\exists z_2\left(x=cy_1^2z_1\wedge z_1\in T_{a,b,K}^\times\wedge x=1-y_2^2z_2\wedge z_2\in T_{a,b,K}^\times\right)$. Here $K^2$ means squares in $K$ and not the cartesian product $K\times K$.

\item $J_{a,b,K}\coloneqq \left(I_{a,b,K}^a+I_{a,b,K}^a\right)\cap \left(I_{a,b,K}^b+I_{a,b,K}^b\right)$, which can be defined by the existential first-order formula $\exists y\exists z\left(y\in I_{a,b,K}^a\wedge x-y\in I_{a,b,K}^a\wedge z\in I_{a,b,K}^b\wedge x-z\in I_{a,b,K}^b\right)$.

\item $J_{a,b,c,d,K}\coloneqq J_{a,b,K}+J_{c,d,K}$, which can be defined by the existential first-order formula $\exists y\left(y\in J_{a,b,K}\wedge x-y\in J_{c,d,K}\right)$.

\end{itemize}

The reason why we introduce these sets is because they are, as mentioned, diophantine in a uniform way with respect to $a,b,c,d$ and, moreover, they allow us to give diophantine definitions for Jacobson radicals of semi-local subrings of $K$. To make this precise and explicit, let us first define:\[\text{If $\lambda\in K^\times$, $\mathbb{P}_K\left(\lambda\right)\coloneqq \left\{\mathfrak{p}\in\Omega_K^{<\infty}:2\nmid \nu_{\mathfrak{p}}\left(\lambda\right)\right\}$},\]\[\Delta_{a,b,K}\coloneqq \left\{v\in \Omega_K:\left(a,b\right)_v=-1\right\},\]\[\Delta^{a,b,K}\coloneqq \Delta_{a,b,K}\cap \left(\mathbb{P}_K\left(a\right)\cup\mathbb{P}_K\left(b\right)\right),\]\[\Omega_{a,b,c,d,K}\coloneqq \Delta^{a,b,K}\cap \Delta^{c,d,K}.\]

Observe that, as a particular case of Lemma \ref{connect} and in light of Section \ref{hsym}, the set $\Delta_{a,b,K}$ is nothing but the set of places $v$ such that $H_{a,b,K_v}$ is split. Moreover, Theorem \ref{epsiv2} implies that $\Delta_{a,b,K}$ (and thus $\Omega_{a,b,c,d,K}$) is a finite set. In Theorem \ref{triunfo} we show that, choosing appropriate $a,b,c,d\in K^\times$, the set $\Omega_{a,b,c,d,K}$ can be any finite subset of $\Omega_K^{<\infty}$, and this will be the main technical result which will make all our arguments work.

In her proof of \cite[Lemma 3.17]{MR3207365} Park proves that if $K$ is a number field and $a,b\in K^\times$ then $\displaystyle{\left(I_{a,b,K}^a+I_{a,b,K}^a\right)\cap \left(I_{a,b,K}^b+I_{a,b,K}^b\right)=\bigcap_{\mathfrak{p}\in\Delta_{a,b,K}\cap \left(\mathbb{P}_K\left(a\right)\cup\mathbb{P}_K\left(b\right)\right)}\mathfrak{p}\left(\mathcal{O}_{K}\right)_\mathfrak{p}}$. Hence:

\begin{lemma}\label{jacobson}If $K$ is a number field and $a,b\in K^\times$ then\[J_{a,b,K}=\bigcap_{\mathfrak{p}\in\Delta^{a,b,K}}\mathfrak{p}\left(\mathcal{O}_{K}\right)_\mathfrak{p}.\]

\end{lemma}

\begin{corollary}\label{jabjab}If $K$ is a number field and $a,b,c,d\in K^\times$ then\[J_{a,b,c,d,K}=\bigcap_{\mathfrak{p}\in\Omega_{a,b,c,d,K}}\mathfrak{p}\left(\mathcal{O}_{K}\right)_\mathfrak{p}.\]

\end{corollary}

We have seen that if $K$ is a number field and $a,b\in K^\times$, the set $J_{a,b,c,d,K}$ is diophantine, and its defining existential first-order formula can be taken to be uniform with respect to $a,b,c,d$. Using this, we can quickly show that the set $\left\{\left(a,b,c,d,r\right)\in \left(K^\times\right)^4\times K:r\in J_{a,b,K}\right\}$ is a diophantine subset of $K^5$ defined by a polynomial with integer coefficients, as witnessed by the formula $\exists x\left(abcdx=1\wedge r\in J_{a,b,c,d,K}\right)$, which makes sense because ``$r\in J_{a,b,c,d,K}$'' is indeed a formula in the variables $a,b,c,d,r$ (because $J_{a,b,c,d,K}$ is defined in a uniform way with respect to $a,b,c,d$), while the existence for a multiplicative inverse of $abcd$ makes sure that $a,b,c,d\in K^\times$. All the existential quantifiers can be moved to the beginning of the formula.

We can generalize this observation to a more general set, which we define next.

\begin{definition}\label{jabn}For each number field $K$ and $a,b,c,d\in K^\times$ we define $J_{a,b,c,d,1,K}\coloneqq J_{a,b,c,d,K}$ and, for each $n\in\mathbb{Z}_{\geq 1}$, $J_{a,b,c,d,n+1,K}\coloneqq J_{a,b,c,d,n,K}\cdot J_{a,b,c,d,K}=\left\{xy:x\in J_{a,b,c,d,n,K}\wedge y\in J_{a,b,c,d,K}\right\}$.

\end{definition}

If $n\in\mathbb{Z}_{\geq 2}$, for a given number field $K$ and $a,b,c,d\in K^\times$, the formula\[\exists y\exists z\left(y\in J_{a,b,c,d,K}\wedge z\in J_{a,b,c,d,K}\wedge x=yz^{n-1}\right)\]defines $J_{a,b,c,d,n,K}$. Indeed, it is immediate that if $x\in K$ is such that the above formula holds, then $x\in J_{a,b,c,d,K}\cdot J_{a,b,c,d,n-1,K}=J_{a,b,c,d,n,K}$. Conversely, assume $x\in J_{a,b,c,d,n,K}$. For each $\mathfrak{p}\in\Omega_{a,b,c,d,K}$, unique factorization gives $\mathfrak{p}^2\subsetneq \mathfrak{p}$, so there exists $z_\mathfrak{p}\in\mathfrak{p}\setminus\mathfrak{p}^2$. Use the Chinese Remainder Theorem to find $z\in \mathcal{O}_{K}$ such that $z\equiv z_\mathfrak{p}\pmod{\mathfrak{p}^2}$ for all $\mathfrak{p}\in \Omega_{a,b,c,d,K}$ (in particular $0\neq z\in J_{a,b,c,d,K}$, and $\nu_\mathfrak{p}\left(z\right)=1$ for all $\mathfrak{p}\in \Omega_{a,b,c,d,K}$). Since $\nu_\mathfrak{p}\left(x\right)\geq n$ (and thus $\nu_\mathfrak{p}\left(\frac{x}{z^{n-1}}\right)\geq n-\left(n-1\right)=1$) for all such $\mathfrak{p}$, then $\frac{x}{z^{n-1}}\in J_{a,b,c,d,K}$, as desired.

In particular, we get:

\begin{lemma}\label{jabndiof}If $K$ is a number field and $n\in\mathbb{Z}_{\geq 1}$, the set\[\left\{\left(a,b,c,d,r\right)\in \left(K^\times\right)^4\times K:r\in J_{a,b,c,d,n,K}\right\}\]is diophantine in $K$, with a defining polynomial with integer coefficients.

\end{lemma}

\section{Parametrization of finite sets of non-archimedean places}\label{pruebadelteom}

We begin with a parametrization of the set of finite subsets of the finite places of a number field, which will allow us to show the uniformity of our first-order definitions over them.

\begin{theorem}\label{triunfo}Fix a number field $K$ and a finite subset $S$ of $\Omega_K^{<\infty}$. Then there exist $a,b,c,d\in K^\times$ such that $\Omega_{a,b,c,d,K}=S$. More precisely, if $S$ has even cardinality, there exist $a,b\in K^\times$ such that $\Delta^{a,b,K}=S$, and if $S$ has odd cardinality, there exist $a,b\in K^\times$ such that $S=\Delta_{a,b,K}\cap\mathbb{P}_K\left(a\right)$.
\end{theorem}

\begin{proof}Assume first $S$ has even cardinality. For each $\mathfrak{p}\in S$, unique factorization of ideals over Dedekind domains gives $\mathfrak{p}^2\subsetneq \mathfrak{p}$, thus we fix $z_\mathfrak{p}\in \mathfrak{p}\setminus\mathfrak{p}^2$ and use the Chinese Remainder Theorem to find $y_\mathfrak{p}\in\mathcal{O}_K$ such that\[\begin{cases}y_\mathfrak{p}\equiv z_\mathfrak{p}\pmod{\mathfrak{p}^2},\\\displaystyle{y_\mathfrak{p}\equiv 1\pmod{\prod_{\mathfrak{q}\in S\setminus\left\{\mathfrak{p}\right\}}\mathfrak{q}}}.\end{cases}\]By construction, if $\mathfrak{p},\mathfrak{q}\in S$ then $\nu_{\mathfrak{p}}\left(y_{\mathfrak{q}}\right)=\delta\left(\mathfrak{p},\mathfrak{q}\right)$, where $\delta$ is the Kronecker function on $S$. If $a\coloneqq \displaystyle{\prod_{\mathfrak{r}\in S}y_\mathfrak{r}}$, it follows that $\nu_\mathfrak{p}\left(a\right)=1$ for all $\mathfrak{p}\in S$.

For each $\mathfrak{p}\in S$, since $\nu_\mathfrak{p}\left(a\right)=1$ then $a$ is not a square in $K_\mathfrak{p}$, thus Proposition \ref{lin} gives $b_\mathfrak{p}\in K_\mathfrak{p}^\times$ such that $\left(a,b_\mathfrak{p}\right)_\mathfrak{p}=-1$. Since $\left(a,1\right)_v=1$ for all $v\in \Omega_K\setminus S$\footnote{Indeed, this holds for all $v\in \Omega_K$.} and $S$ is finite and of even cardinality, Theorem \ref{epsiv2} applies and gives $b\in K^\times$ such that $\left(a,b\right)_v=1$ for all $v\in \Omega_K\setminus S$ and $\left(a,b\right)_\mathfrak{p}=-1$ for all $\mathfrak{p}\in S$; that is, $\Delta_{a,b,K}=S$. Moreover, since $\nu_\mathfrak{p}\left(a\right)=1$ for all $\mathfrak{p}\in S$ then $\Delta_{a,b,K}\subseteq \mathbb{P}_K\left(a\right)$, thus\[\Delta_{a,b,K}=\Delta_{a,b,K}\cap \mathbb{P}_K\left(a\right)\subseteq \Delta^{a,b,K}=\Delta_{a,b,K}\cap \left(\mathbb{P}_K\left(a\right)\cup\mathbb{P}_K\left(b\right)\right)\subseteq \Delta_{a,b,K},\]giving $\Delta^{a,b,K}=\Delta_{a,b,K}=S$.

Now assume $S$ has odd cardinality, and fix a $\mathfrak{q}\in \Omega_K^{<\infty}\setminus S$ not lying above $2$. Define $S'\coloneqq S\cup\left\{\mathfrak{q}\right\}$. As before, for each $\mathfrak{p}\in S'$ fix $y_\mathfrak{p}\in\mathcal{O}_K$ such that $y_\mathfrak{p}\equiv z_\mathfrak{p}\pmod{\mathfrak{p}^2}$ and $\displaystyle{y_\mathfrak{p}\equiv 1\pmod{\prod_{\mathfrak{r}\in S'\setminus\left\{\mathfrak{p}\right\}}\mathfrak{r}}}$, where $z_\mathfrak{p}\in \mathfrak{p}\setminus\mathfrak{p}^2$. Hence $\nu_{\mathfrak{p}}\left(y_{\mathfrak{r}}\right)=\delta\left(\mathfrak{p},\mathfrak{r}\right)$ for all $\mathfrak{p},\mathfrak{r}\in S'$. It follows that if $a\coloneqq \displaystyle{\prod_{\mathfrak{r}\in S'}y_\mathfrak{r}}$ then $\nu_\mathfrak{p}\left(a\right)=1$ for all $\mathfrak{p}\in S'$.

Let $q$ be the cardinality of $\mathbb{F}_\mathfrak{q}$ (so $q$ is odd because $\mathfrak{q}$ does not lie above $2$), and take a generator $g$ of the cyclic group $\mathbb{F}_\mathfrak{q}^\times$, thus $g^{\frac{q-1}{2}}=-1$. Lift $g$ to some $\widetilde{g}\in \mathcal{O}_{K}\setminus \mathfrak{q}\mathcal{O}_K\subseteq\mathcal{O}_{K,\mathfrak{q}}^\times$, and use the Chinese Remainder Theorem to find a $\tau\in\mathcal{O}_K$ such that $\tau\frac{a}{y_\mathfrak{q}}\equiv \widetilde{g}\pmod{\mathfrak{q}}$ (which has a solution because $\nu_\mathfrak{q}\left(\frac{a}{y_\mathfrak{q}}\right)=1-1=0$) and $\displaystyle{\tau\equiv 1\pmod{\prod_{\mathfrak{r}\in S}\mathfrak{r}}}$.

By construction, if we define $a'\coloneqq \tau y_\mathfrak{q}a$ we get $\nu_\mathfrak{p}\left(a'\right)=1$ for all $\mathfrak{p}\in S$ and $\nu_\mathfrak{q}\left(a'\right)=2$. The former implies that $a'$ is not a square in $K_\mathfrak{p}$ for all $\mathfrak{p}\in S$. I claim that $a'$ is also a nonsquare in $K_\mathfrak{q}$. Indeed, writing $a'=\tau \frac{a}{y_\mathfrak{q}}y_\mathfrak{q}^2$, this is equivalent to showing that $\tau \frac{a}{y_\mathfrak{q}}$ is a nonsquare in $K_\mathfrak{q}$. But by our definition of $\tau$ we get $\tau \frac{a}{y_\mathfrak{q}}\equiv \widetilde{g}\pmod{\mathfrak{q}}$, thus if $\tau \frac{a}{y_\mathfrak{q}}$ were a square we would get $-1=g^{\frac{q-1}{2}}=1$ (the latter equality being given by Lagrange's theorem on $\mathbb{F}_\mathfrak{q}^\times$), giving $2=0\in\mathbb{F}_\mathfrak{q}$. This contradicts that $\mathfrak{q}$ does not lie above $2$.

Apply Proposition \ref{lin} to find $b_\mathfrak{p}\in K_\mathfrak{p}^\times$ such that $\left(a',b_\mathfrak{p}\right)_\mathfrak{p}=-1$ for all $\mathfrak{p}\in S'$, and since $S'$ is finite of even cardinality, we apply Theorem \ref{epsiv2} and find $b'\in K^\times$ such that $\Delta_{a',b',K}=S'$. Since $\nu_\mathfrak{p}\left(a'\right)=1$ for all $\mathfrak{p}\in S$ and $\nu_\mathfrak{q}\left(a'\right)=2$ then $S=S'\cap\mathbb{P}_K\left(a'\right)=\Delta_{a',b',K}\cap\mathbb{P}_K\left(a'\right)$, completing the proof. $\blacksquare$

\end{proof}

Theorem \ref{triunfo}, together with our uniform existential definitions in Section \ref{preliminares} will allow us to construct diophantine sets with arithmetic significance and obtain our desired descriptions of Campana points. A first step in this direction is to have some control on the elements of a number field having negative valuation with respect to a given finite set of non-archimedean places, and the following lemma will allow us to give such a description.

\begin{lemma}\label{pirueta}Let $K$ be a number field and let $a,b,c,d\in K^\times$. Then for all $n\in\mathbb{Z}_{\geq 1}$ we have\[\left(J_{a,b,c,d,n,K}\setminus\left\{0\right\}\right)^{-1}=\left\{r\in K^\times:\forall \mathfrak{p}\in\Omega_{a,b,c,d,K}\left(\nu_\mathfrak{p}\left(r\right)\leq -n\right)\right\}.\]

\end{lemma}

\begin{proof}Since taking inverses is a bijection, it suffices to show that\begin{equation}\label{formmm}J_{a,b,c,d,n,K}\setminus\left\{0\right\}=\left\{x\in K^\times:\forall \mathfrak{p}\in\Omega_{a,b,c,d,K}\left(\nu_\mathfrak{p}\left(x\right)\geq n\right)\right\}\end{equation}for all $n\in\mathbb{Z}_{\geq 1}$. We argue by induction. The case $n=1$ follows directly from Corollary \ref{jabjab}, thus fix $n\in\mathbb{Z}_{\geq 1}$ such that \eqref{formmm} is true.

If $x\in J_{a,b,c,d,n+1,K}\setminus\left\{0\right\}$ then $x=yz$ for some $y\in J_{a,b,c,d,n,K}$ and $z\in J_{a,b,c,d,K}$, thus by inductive hypothesis and our base step, $\nu_\mathfrak{p}\left(x\right)=\nu_\mathfrak{p}\left(y\right)+\nu_{\mathfrak{p}}\left(z\right)\geq n+1$ for all $\mathfrak{p}\in \Omega_{a,b,c,d,K}$. Conversely, assume $x\in K^\times$ and $\nu_\mathfrak{p}\left(x\right)\geq n+1$ for all $\mathfrak{p}\in\Omega_{a,b,c,d,K}$. For each such $\mathfrak{p}$ we can fix $z_\mathfrak{p}\in\mathfrak{p}\setminus\mathfrak{p}^2$.  Chinese Remainder Theorem gives $z\in \mathcal{O}_K$ such that $z\equiv z_\mathfrak{p}\pmod{\mathfrak{p}^2}$ for all $\mathfrak{p}\in \Omega_{a,b,c,d,K}$. So $\nu_\mathfrak{p}\left(z\right)=1$ for all $\mathfrak{p}\in \Omega_{a,b,c,d,K}$ and thus $z\in J_{a,b,c,d,K}$. If $y\coloneqq xz^{-1}$ then $\nu_{\mathfrak{p}}\left(y\right)=\nu_{\mathfrak{p}}\left(x\right)-\nu_{\mathfrak{p}}\left(z\right)=\nu_{\mathfrak{p}}\left(x\right)-1\geq \left(n+1\right)-1=n$ for all $\mathfrak{p}\in\Delta^{a,b,K}$.

By inductive hypothesis $y\in J_{a,b,c,d,n,K}$, so $x=yz\in J_{a,b,c,d,n,K}\cdot J_{a,b,c,d,K}=J_{a,b,c,d,n+1,K}$. $\blacksquare$

\end{proof}

\begin{proposition}\label{tres}If $K$ is a number field and $n\in\mathbb{Z}_{\geq 1}$, then the set\[\left\{\left(a,b,c,d,r\right)\in \left(K^\times\right)^4\times K:r\in \left(J_{a,b,c,d,n,K}\setminus\left\{0\right\}\right)^{-1}\right\}\]is diophantine in $K$.
\end{proposition}

\begin{proof}Fix $n\in\mathbb{Z}_{\geq 1}$. By Lemma \ref{jabndiof}, the formula (in the free variables $a,b,c,d,r$)\[\exists x\exists y\left(abcdx=1\wedge ry=1\wedge y\in J_{a,b,n,K}\right)\]gives our desired existential definition. $\blacksquare$
\end{proof}

Another conceptual use of Theorem \ref{triunfo} is that, given a number field $K$, we can give a diophantine definition for the property that two finite subsets of $\Omega_K^{<\infty}$ have empty intersection.

\begin{proposition}\label{intvacia}Given a number field $K$, the set\[\left\{\left(a,b,c,d,a',b',c',d'\right)\in \left(K^\times\right)^8:\Omega_{a,b,c,d,K}\cap\Omega_{a',b',c',d',K}=\emptyset\right\}\]is diophantine in $K$, with a defining polynomial with integer coefficients.

\end{proposition}

\begin{proof}If $a,b,c,d,a',b',c',d'\in K^\times$, we have\[J_{a,b,c,d,K}+J_{a',b',c',d',K}=\bigcap_{\mathfrak{p}\in\Omega_{a,b,c,d,K}\cap \Omega_{a',b',c',d',K}}\mathfrak{p}\left(\mathcal{O}_K\right)_\mathfrak{p},\]thus $\Omega_{a,b,c,d,K}\cap\Omega_{a',b',c',d',K}=\emptyset$ if and only if $1\in J_{a,b,c,d,K}+J_{a',b',c',d',K}$. Lemma \ref{jabndiof} with $n=1$ shows that the formula in the free variables $a,b,c,d,a',b',c',d'$\[\exists x\exists y\left(abcda'b'c'd'x=1\wedge y\in J_{a,b,c,d,K}\wedge 1-y\in J_{a',b',c',d',K}\right)\]gives our desired existential definition. $\blacksquare$

\end{proof}

\section{First-order description of Campana points}\label{primerdescripcion}

We can now put the above results together to produce our desired first-order description of Campana points.

\begin{theorem}\label{campaanaa}Given a number field $K$, the set\[\left\{\left(a,b,c,d,r\right)\in \left(K^\times\right)^4\times K:r\in C_{K,\left(\Omega_K^\infty\cup\Omega_{a,b,c,d,K}\right),n}\right\}\cup \left\{\left(x_1,x_2,x_3,x_4,x_5\right)\in K^5:x_1x_2x_3x_4=0\right\}\]is $\forall\exists$-defined in $K$, with a defining polynomial with integer coefficients.

\end{theorem}

\begin{proof}If $r,a,b,c,d\in K^\times$, then $r\in C_{K,\left(\Omega_K^\infty\cup\Omega_{a,b,c,d,K}\right),n}$ if and only if each $\mathfrak{p}\in\Omega_K^{<\infty}\setminus \Omega_{a,b,c,d,K}$ dividing a reduced denominator of $r\in K_{\mathfrak{p}}$, divides it with valuation at least $n$.

Observe that, by Theorem \ref{triunfo}, any finite set of non-archimedean places of $K$ can be attained as $\Omega_{a',b',c',d',K}$ for some $a',b',c',d'\in K^\times$. Thus an equivalent reformulation of the previous paragraph is the following: for all $a',b',c',d'\in K^\times$ such that $\Omega_{a,b,c,d,K}\cap \Omega_{a',b',c',d',K}=\emptyset$, if $r\in \left(J_{a',b',c',d',K}\setminus\left\{0\right\}\right)^{-1}$ (i.e., if $\mathfrak{p}\in\Omega_{a',b',c',d',K}$, a reduced denominator of $r\in K_{\mathfrak{p}}$ is divisible by $\mathfrak{p}$, by Corollary \ref{jabjab}), then $r\in \left(J_{a',b',c',d',n,K}\setminus\left\{0\right\}\right)^{-1}$ (i.e., reduced denominators of $r$ with respect to all primes in $\Omega_{a',b',c',d',K}$ have valuations at least $n$, by Lemma \ref{pirueta}). This can all be expressed by the first-order formula\[\forall a'\forall b'\forall c'\forall d'\left[\begin{pmatrix}abcda'b'c'd'\neq 0\\ \Omega_{a,b,c,d,K}\cap\Omega_{a',b',c',d',K}=\emptyset\\r\in \left(J_{a',b',c',d',K}\setminus\left\{0\right\}\right)^{-1}\end{pmatrix}\Rightarrow r\in \left(J_{a',b',c',d',n,K}\setminus\left\{0\right\}\right)^{-1}\right],\]\begin{equation}\label{campanadef2}\forall a'\forall b'\forall c'\forall d'\left[\neg\begin{pmatrix}abcda'b'c'd'\neq 0\\ \Omega_{a,b,c,d,K}\cap\Omega_{a',b',c',d',K}=\emptyset\\r\in \left(J_{a',b',c',d',K}\setminus\left\{0\right\}\right)^{-1}\end{pmatrix}\vee r\in \left(J_{a',b',c',d',n,K}\setminus\left\{0\right\}\right)^{-1}\right].\end{equation}In this formula, what is inside the main brackets is the disjunction of a universal formula (by Lemma \ref{intofdiof} and Propositions \ref{tres} and \ref{intvacia}) and an existential formula (by Proposition \ref{tres}). In general, if $k$ is a field and $P\in k\left[\mathbf{x}\right]$, $Q\in k\left[\mathbf{z}\right]$, then it is straightforward to prove\begin{equation}\label{ae}k\vDash \left[\left(\forall \mathbf{x} \left(P\left(\mathbf{x}\right)\neq 0\right)\right)\vee \left(\exists \mathbf{z}\left(Q\left(\mathbf{z}\right)=0\right)\right)\right]\Longleftrightarrow \left[\forall\mathbf{x}\exists y\exists\mathbf{z}\left(\left(yP\left(\mathbf{x}\right)-1\right)Q\left(\mathbf{z}\right)=0\right)\right],\end{equation}so the formula \eqref{campanadef2} is indeed a $\forall\exists$-definition for the set in the statement. $\blacksquare$

\end{proof}

In particular, given a number field $K$, a natural number $n\in\mathbb{Z}_{\geq 1}$, and a finite subset $S$ of $\Omega_K$ containing $\Omega_K^{\infty}$, $C_{K,S,n}$ can be $\forall\exists$-defined as follows: use Theorem \ref{triunfo} to pick $a,b,c,d\in K^\times$ such that $\Omega_{a,b,c,d,K}=S\cap\Omega_K^{<\infty}$. Then our set is defined by the formula\[\forall a'\forall b'\forall c'\forall d'\left[\begin{pmatrix}abcda'b'c'd'\neq 0\\ \Omega_{a,b,c,d,K}\cap\Omega_{a',b',c',d',K}=\emptyset\\r\in \left(J_{a',b',c',d',K}\setminus\left\{0\right\}\right)^{-1}\end{pmatrix}\Rightarrow r\in \left(J_{a',b',c',d',n,K}\setminus\left\{0\right\}\right)^{-1}\right],\]\[\forall a'\forall b'\forall c'\forall d'\left[\neg\begin{pmatrix}abcda'b'c'd'\neq 0\\ \Omega_{a,b,c,d,K}\cap\Omega_{a',b',c',d',K}=\emptyset\\r\in \left(J_{a',b',c',d',K}\setminus\left\{0\right\}\right)^{-1}\end{pmatrix}\vee r\in \left(J_{a',b',c',d',n,K}\setminus\left\{0\right\}\right)^{-1}\right].\]The difference between this formula and the one in the proof of Theorem \ref{campaanaa} is that this one has only one free variable (namely, $r$), and $a,b,c,d\in K^\times$ are fixed nonzero elements of $K$. In particular, the only parameters that may arise in this formula are $a,b,c,d$.

We then have the following immediate weak version of the above:

\begin{theorem}For all $n\in\mathbb{Z}_{\geq 1}$, the set $C_{K,S,n}$ is $\forall\exists$-definable in $K$.

\end{theorem}

Observe that the same method can be applied to get a $\forall\exists$-definition of $S$-integers of a number field $K$ (for a given finite subset $S$ of $\Omega_K$ containing $\Omega_K^{\infty}$) in a uniform way, with respect to our parametrization given in Theorem \ref{triunfo}. This recovers the main result in \cite{MR2530851}.

\begin{theorem}\label{integral}Given a number field $K$, the set\begin{multline*}\left\{\left(a,b,c,d,r\right)\in \left(K^\times\right)^4\times K:r\in \bigcap_{\mathfrak{p}\in\Omega_K^{<\infty}\setminus\Omega_{a,b,c,d,K}}\left(\mathcal{O}_{K}\right)_\mathfrak{p}\right\}\\\cup\left\{\left(x_1,x_2,x_3,x_4,x_5\right)\in K^5:x_1x_2x_3x_4=0\right\}\end{multline*}is $\forall\exists$-defined in $K$, with a defining polynomial with integer coefficients.

\end{theorem}

\begin{proof}By Corollary \ref{jabjab}, the predication ``$r\not\in \left(J_{a,b,c,d,K}\setminus\left\{0\right\}\right)^{-1}$'' (where $a,b,c,d\in K^\times$) means that the valuations of $r$ with respect to all primes in $\Omega_{a,b,c,d,K}$ are not all negative. If we quantify over all $a,b,c,d\in K^\times$, Theorem \ref{triunfo} guarantees that all finite subsets of $\Omega_K^{<\infty}$ will arise in this way; in particular, singletons are attained, and that is all that we need: all the rest, excepting those $a,b,c,d\in K^\times$ giving $\Omega_{a,b,c,d,K}=\emptyset$ (which by Corollary \ref{jabjab} can be first-order defined as $1\in J_{a,b,c,d,K}$), will give redundant information. Thus a formula in the free variables $a,b,c,d,r$ defining the desired set is\[\forall a'\forall b'\forall c'\forall d'\left(\left[\substack{abcda'b'c'd'\neq 0\\\Omega_{a,b,c,d,K}\cap\Omega_{a',b',c',d',K}=\emptyset}\wedge 1\not\in J_{a',b',c',d',K}\right]\Rightarrow r\not\in \left(J_{a',b',c',d',K}\setminus\left\{0\right\}\right)^{-1}\right),\]\[\forall a'\forall b'\forall c'\forall d'\left(\neg\left[\substack{abcda'b'c'd'\neq 0\\\Omega_{a,b,c,d,K}\cap\Omega_{a',b',c',d',K}=\emptyset}\wedge 1\not\in J_{a',b',c',d',K}\right]\vee r\not\in \left(J_{a',b',c',d',K}\setminus\left\{0\right\}\right)^{-1}\right),\]which a $\forall\exists$-formula by Proposition \ref{tres}, Lemma \ref{jabndiof}, and Proposition \ref{intvacia}. $\blacksquare$

\end{proof}

\begin{corollary}\label{integrality}Given any number field $K$ and a finite subset $S$ of $\Omega_K$ containing $\Omega_K^{\infty}$, the set $\mathcal{O}_{K,S}$ of elements of $K$ which are integral outside $S$ is $\forall\exists$-defined in $K$.

\end{corollary}

\begin{proof}Use Theorem \ref{triunfo} to pick $a,b,c,d\in K^\times$ such that $\Omega_{a,b,c,d,K}=S\cap\Omega_K^{<\infty}$. Then the formula\[\forall a'\forall b'\forall c'\forall d'\left(\left[\substack{abcda'b'c'd'\neq 0\\\Omega_{a,b,c,d,K}\cap\Omega_{a',b',c',d',K}=\emptyset}\wedge 1\not\in J_{a',b',c',d',K}\right]\Rightarrow r\not\in \left(J_{a',b',c',d',K}\setminus\left\{0\right\}\right)^{-1}\right),\]\[\forall a'\forall b'\forall c'\forall d'\left(\neg\left[\substack{abcda'b'c'd'\neq 0\\\Omega_{a,b,c,d,K}\cap\Omega_{a',b',c',d',K}=\emptyset}\wedge 1\not\in J_{a',b',c',d',K}\right]\vee r\not\in \left(J_{a',b',c',d',K}\setminus\left\{0\right\}\right)^{-1}\right),\]defines $\mathcal{O}_{K,S}$ in $K$. $\blacksquare$

\end{proof}

As before, the difference between the formula considered in the proof of Theorem \ref{integral} and the one considered in the proof of Corollary \ref{integrality} is that the latter one has only one free variable (namely, $r$), and $a,b,c,d\in K^\times$ are fixed nonzero elements of $K$. In particular, this formula may not be given in terms of a defining polynomial with integer coefficients, although its only parameters will be (at most) $a,b,c,d$.

\section{Counting variables and bounding the degree}\label{explicito}

We now return to the formulas in Section \ref{primerdescripcion} and we count the number of quantifiers (that is, bounded variables) involved in a possible defining formula, and we also bound the degree of the considered polynomial. In order to do that, we will make use of normic homogeneous forms (see \cite[§1.2.1]{MR3729254}) to optimize Remark \ref{bounddegree}. We will need the following straightforward generalization of the lemma in \cite[p. 34]{MR0457396}:

\begin{lemma}\label{mmarcus}Let $k$ be a field and let $K/k$ and $L/k$ be two finite separable extensions of fields. If $\left[KL:k\right]=\left[K:k\right]\left[L:k\right]$, then given two field embeddings $K\to \overline{k}$ and $L\to\overline{k}$ fixing $k$ pointwise, there exists a field embedding $KL\to \overline{k}$ extending both.
\end{lemma}

We will also need the following two facts.

\begin{lemma}\label{ksqrtm}Let $K$ be a number field. For all but finitely many primes $p\in\mathbb{Z}_{\geq 1}$, we have $\left[K\left(\sqrt[n]{p}\right):K\right]=n$ for all $n\in\mathbb{Z}_{\geq 1}$.
\end{lemma}

\begin{proof}Fix a prime $p\in\mathbb{Z}_{\geq 1}$ such that $\nu_\mathfrak{P}\left(p\mathcal{O}_K\right)=1$ for some prime ideal $\mathfrak{P}$ of $\mathcal{O}_K$ (for instance, take any of the all but finitely many primes in $\mathbb{Z}_{\geq 1}$ which are unramified in $K$). By Eisensteins's criterion for Dedekind domains, $x^n-p$ is irreducible in $K\left[x\right]$ for all $n\in\mathbb{Z}_{\geq 1}$. $\blacksquare$

\end{proof}

\begin{lemma}\label{agregar}Fix a field $k$ and $n\in\mathbb{Z}_{\geq 1}$. Let $L/k$ be a finite separable extension of fields. If $y_1,\cdots,y_n$ are algebraically independent indeterminates, we have\[\left[L:k\right]=\left[L\left(y_1,\cdots,y_n\right):k\left(y_1,\cdots,y_n\right)\right].\]
\end{lemma}

\begin{proof}By \cite[Chapter V, Proposition 1.7]{MR1878556}, since $L/k$ is algebraic, $L\left(y_1,\cdots,y_n\right)/k\left(y_1,\cdots,y_n\right)$ is algebraic as well. Moreover, it is separable, because it is generated by $L$, which is separable over $k$ (thus over $k\left(y_1,\cdots,y_n\right)\supseteq k$) by hypothesis.

Given a field embedding $\sigma:L\to \overline{k}$ fixing $k$ pointwise, define a ring homomorphism\[\widehat{\sigma}:L\left[y_1,\cdots,y_n\right]\to \overline{k\left(y_1,\cdots,y_n\right)}\]by extending $\sigma$ and fixing each indeterminate. It has trivial kernel, so it extends to a field embedding $L\left(y_1,\cdots,y_n\right)\to \overline{k\left(y_1,\cdots,y_n\right)}$ fixing $k\left(y_1,\cdots,y_n\right)$ pointwise. Two different field embeddings $L\to \overline{k}$ fixing $k$ pointwise produce, by this method, different field embeddings $L\left(y_1,\cdots,y_n\right)\to \overline{k\left(y_1,\cdots,y_n\right)}$ fixing $k\left(y_1,\cdots,y_n\right)$ pointwise. Since there are exactly $\left[L:k\right]$ field embeddings fixing $k$ pointwise $L\to \overline{k}$ because $L/k$ is separable, we get\[\left[L\left(y_1,\cdots,y_n\right):k\left(y_1,\cdots,y_n\right)\right]\geq \left[L:k\right].\]

Conversely, by the Primitive Element Theorem (\cite[Chapter V, Theorem 4.6]{MR1878556}) $E=k\left(\alpha\right)$ for some $\alpha\in L$ whose minimal polynomial over $k$ has degree $\left[L:k\right]$. Since this minimal polynomial divides the minimal polynomial of $\alpha$ over $k\left(y_1,\cdots,y_n\right)$ and $L\left(y_1,\cdots,y_n\right)=k\left(y_1,\cdots,y_n\right)\left(\alpha\right)$, we get $\left[L\left(y_1,\cdots,y_n\right):k\left(y_1,\cdots,y_n\right)\right]\leq \left[L:k\right]$. $\blacksquare$

\end{proof}

Finally, we prove the following improvement of Remark \ref{bounddegree}:

\begin{proposition}\label{intdiofopt}Let $K$ be a number field and fix $m,n\in\mathbb{Z}_{\geq 1}$. For each $j\in\mathbb{Z}_{\leq n}$ fix a nonzero $f_j\in K\left[x_1,\cdots,x_m\right]$. Let $\displaystyle{d\coloneqq \max_{j\in\left[1,n\right]\cap\mathbb{Z}}\deg\left(f_j\right)}$. Then there exists $F\in K\left[x_1,\cdots,x_m\right]$ such that $\deg\left(F\right)\leq dn$ and\[K\vDash \forall x_1\cdots \forall x_m\left[\left(\bigwedge_{j=1}^nf_j\left(x_1,\cdots,x_m\right)=0\right)\Longleftrightarrow F\left(x_1,\cdots,x_m\right)=0\right].\]Moreover, by cleaning denominators, the only parameters that may appear in the above first-order formula are the coefficients of $f_j$ for all $j\in\left[1,n\right]\cap\mathbb{Z}$.

\end{proposition}

\begin{proof}Use Lemma \ref{ksqrtm} to fix a prime number $p\in\mathbb{Z}_{\geq 1}$ such that $\left[K\left(\sqrt[n]{p}\right):K\right]=n$. Since $x^n-p\in\mathbb{Q}\left[x\right]$ is irreducible (by Eisenstein's criterion), $\mathbb{Q}\left(\sqrt[n]{p}\right)$ has degree $n$ over $\mathbb{Q}$. We thus have the diagram\begin{center}
\begin{tikzcd}
                           & {K\left(\sqrt[n]{p}\right)}                                                             &                                                                                          \\
K \arrow[ru, "n", no head] &                                                                                         & {\mathbb{Q}\left(\sqrt[n]{p}\right)} \arrow[lu, "{\left[K:\mathbb{Q}\right]}"', no head] \\
                           & \mathbb{Q} \arrow[lu, "{\left[K:\mathbb{Q}\right]}", no head] \arrow[ru, "n"', no head] &                                                                                         
\end{tikzcd}\end{center}which by Lemma \ref{agregar} induces the diagram\begin{center}
\begin{tikzcd}
                                                        & {K\left(\sqrt[n]{p}\right)\left(y_1,\cdots,y_n\right)}                                                               &                                                                                                                     \\
{K\left(y_1,\cdots,y_n\right)} \arrow[ru, "n", no head] &                                                                                                                      & {\mathbb{Q}\left(\sqrt[n]{p}\right)\left(y_1,\cdots,y_n\right)} \arrow[lu, "{\left[K:\mathbb{Q}\right]}"', no head] \\
                                                        & {\mathbb{Q}\left(y_1,\cdots,y_n\right)} \arrow[lu, "{\left[K:\mathbb{Q}\right]}", no head] \arrow[ru, "n"', no head] &                                                                                                                    
\end{tikzcd}\end{center}where $y_1,\cdots,y_n$ are algebraically independent indeterminates.

Define\[G\left(y_1,\cdots,y_n\right)\coloneqq \nor_{\mathbb{Q}\left(y_1,\cdots,y_n\right)}^{\mathbb{Q}\left(\sqrt[n]{p}\right)\left(y_1,\cdots,y_n\right)}\left(\sum_{j=1}^{n}y_j\sqrt[n]{p}^{j-1}\right)\in \mathbb{Q}\left[y_1,\cdots,y_n\right],\]which is an homogeneous polynomial over $\mathbb{Q}$ of degree $n$ whose only zero in the rationals is the trivial one. Define\[F\left(x_1,\cdots,x_m\right)\coloneqq G\left(f_1\left(x_1,\cdots,x_m\right),\cdots,f_n\left(x_1,\cdots,x_m\right)\right),\]whose coefficients are where we want them to be and whose degree is at most $nd$.

To complete the proof, let $a_1,\cdots,a_n\in K$ be such that $G\left(a_1,\cdots,a_n\right)=0$, and let us show that $a_j=0$ for all $j\in\left[1,n\right]\cap\mathbb{Z}$ (the converse is immediate). By Lemma \ref{mmarcus}, each field embedding $\sigma:\mathbb{Q}\left(\sqrt[n]{p}\right)\left(y_1,\cdots,y_n\right)\to \overline{\mathbb{Q}\left(y_1,\cdots,y_n\right)}$ fixing $\mathbb{Q}\left(y_1,\cdots,y_n\right)$ pointwise extends to a field embedding \mbox{$\widetilde{\sigma}:K\left(\sqrt[n]{p}\right)\left(y_1,\cdots,y_n\right)\to \overline{\mathbb{Q}\left(y_1,\cdots,y_n\right)}$} fixing $K\left(y_1,\cdots,y_n\right)$ pointwise. We therefore obtain\[0=G\left(a_1,\cdots,a_n\right)=\prod_{\sigma\in\mathfrak{G}}\left(\sum_{j=1}^na_j\sigma\left(\sqrt[n]{p}^{j-1}\right)\right)=\prod_{\sigma\in \mathfrak{G}}\widetilde{\sigma}\left(\sum_{j=1}^{n}a_j\sqrt[n]{p}^{j-1}\right),\]where $\mathfrak{G}$ is the set of all field embeddings $\mathbb{Q}\left(\sqrt[n]{p}\right)\left(y_1,\cdots,y_n\right)\to \overline{\mathbb{Q}\left(y_1,\cdots,y_n\right)}$ fixing $\mathbb{Q}\left(y_1,\cdots,y_n\right)$ pointwise. This implies $\displaystyle{\sum_{j=1}^{n}a_j\sqrt[n]{p}^{j-1}=0}$, so the polynomial $\displaystyle{g\left(x\right)\coloneqq\sum_{j=1}^{n}a_jx^{j-1}\in K\left[x\right]}$ vanishes at $\sqrt[n]{p}$. If it were nonzero, then $n-1\geq \left[K\left(\sqrt[n]{p}\right):K\right]=n$, clearly impossible. Thus $g\left(x\right)\in K\left[x\right]$ is the zero polynomial, which by definition means $a_j=0$ for all $j\in\left[1,n\right]\cap\mathbb{Z}$. $\blacksquare$

\end{proof}

We are now in a position to state and prove the following concrete statement that adds computational details to Theorem \ref{campaanaa} by specifying the number of involved quantifiers and a bound for the degree of the defining polynomial.

\begin{proposition}\label{volver}If $K$ is a number field and $n\in\mathbb{Z}_{\geq 2}$, there exists\[P\in \mathbb{Z}\left[A,B,C,D,R,X_1,\cdots,X_{838},Y_1,\cdots,Y_{558}\right]\]of degree at most $\max\left\{194n+2611,3387\right\}$ (if $K\subseteq\mathbb{R}$, this degree is at most $\max\left\{2n+19,27\right\}$) such that, for any $a,b,c,d,r\in K$, the following are equivalent:

\begin{itemize}

\item $abcd=0$, or $abcd\neq 0$ and $\nu_\mathfrak{p}\left(r\right)\in\mathbb{Z}_{\geq 0}\cup\mathbb{Z}_{\leq -n}$ for all $\mathfrak{p}\in\Omega_K^{<\infty}\setminus \Omega_{a,b,c,d,K}$.

\item For all $x_1,\cdots,x_{838}\in K$ there exist $y_1,\cdots,y_{558}\in K$ such that

$P\left(a,b,c,d,r,x_1,\cdots,x_{838},y_1,\cdots,y_{558}\right)=0$.

\end{itemize}

\end{proposition}

\begin{proof}We start by looking at our definitions in Section \ref{imppdiof}. First observe that, given $a,b\in K^\times$, the set $S_{a,b,K}$ can be defined by the first-order formula in the free variable $r$\[\exists x_2\exists x_3\exists x_4\left(r^2-4ax_2^2-4bx_3^2+4abx_4^2=4\right),\]which involves $3$ quantifiers and a polynomial of degree $4$ (including the degree contributions of $a$ and $b$).

Given $a,b\in K^\times$, the set $T_{a,b,K}$ can be defined by\[\exists x\left(x\in S_{a,b,K}\wedge r-x\in S_{a,b,K}\right),\]which involves $1+2\cdot 3=7$ quantifiers and can be rewritten as\[\exists x_1\cdots \exists x_7\left(\varphi_{2,4}\left(a,b,r,x_1,\cdots,x_7\right)\right),\]where $\varphi_{2,4}\left(a,b,r,x_1,\cdots,x_7\right)$ is the conjunction of $2$ polynomial equalities, each of degree $4$, and having integer coefficients.

Given $a,b\in K^\times$, the set $T_{a,b,K}^\times$ can be defined by\[r\in T_{a,b,K}\wedge \exists v\left(v\in T_{a,b,K}\wedge \,rv=1\right),\]which involves $1+2\cdot 7=15$ quantifiers and can thus be rewritten as\[\exists x_1\cdots \exists x_{15}\left(\varphi_{5,4}\left(a,b,r,x_1,\cdots,x_{15}\right)\right),\]where $\varphi_{5,4}\left(a,b,r,x_1,\cdots,x_{15}\right)$ is the conjunction of $2\cdot 2+1=5$ polynomial equalities of degree at most $\max\left\{4,2\right\}=4$ and having integer coefficients.

Given $a,b,c\in K^\times$, the set $I_{a,b,K}^c$ can be defined by\[\exists x\exists y\exists u\exists v\left(u\in T_{a,b,K}^\times\wedge v\in T_{a,b,K}^\times\wedge r=cx^2u\wedge r=1-y^2v\right),\]which involves $4+2\cdot 15=34$ quantifiers. We rewrite this as\[\exists x_1\cdots \exists x_{34}\left(\varphi_{12,4}\left(a,b,r,x_1,\cdots,x_{34}\right)\right),\]where $\varphi_{12,4}\left(a,b,r,x_1,\cdots,x_{34}\right)$ is the conjunction of $2\cdot 5+1+1=12$ polynomial equalities with integer coefficients, each of degree at most $\max\left\{4,3\right\}=4$, and having integer coefficients.

Given $a,b\in K^\times$, the set $J_{a,b,K}$ can be defined by\[\exists x\exists y\left(x\in I_{a,b,K}^a\wedge r-x\in I_{a,b,K}^a\wedge y\in I_{a,b,K}^b\wedge r-y\in I_{a,b,K}^b\right),\]which involves $2+4\cdot 34=138$ quantifiers. Rewrite this as\[\exists x_1\cdots \exists x_{138}\left(\varphi_{48,4}\left(a,b,r,x_1,\cdots,x_{138}\right)\right),\]where $\varphi_{48,4}\left(a,b,r,x_1,\cdots,x_{138}\right)$ is the conjunction of $4\cdot 12=48$ polynomial equalities of degree at most $4$ and having integer coefficients.

Given $a,b,c,d\in K^\times$, the set $J_{a,b,c,d,K}$ can be defined by\[\exists x\left(x\in J_{a,b,K}\wedge r-x\in J_{c,d,K}\right),\]which involves $1+2\cdot 138=277$ quantifiers. Rewrite as\[\exists x_1\cdots \exists x_{277}\left(\varphi_{96,4}\left(a,b,c,d,r,x_1,\cdots,x_{277}\right)\right),\]where $\varphi_{96,4}\left(a,b,c,d,r,x_1,\cdots,x_{277}\right)$ is the conjunction of $2\cdot 48=96$ polynomial equalities of degree at most $4$ and having integer coefficients.

Given $a,b,c,d\in K^\times$, the set $\left(J_{a,b,c,d,K}\setminus\left\{0\right\}\right)^{-1}$ can be defined by\[\exists y\left(y\in J_{a,b,c,d,K}\wedge \,ry=1\right),\]which involves $1+277=278$ quantifiers. We can then rewrite this as\[\exists x_1\cdots \exists x_{278}\left(\varphi_{97,4}\left(a,b,c,d,r,x_1,\cdots,x_{278}\right)\right),\]where $\varphi_{97,4}\left(a,b,c,d,r,x_1,\cdots,x_{278}\right)$ is the conjunction of $96+1=97$ polynomial equalities of degree at most $\max\left\{4,2\right\}=4$ and having integer coefficients.

The set $\left\{\left(a,b,c,d,a',b',c',d'\right)\in \left(K^\times\right)^8:\Omega_{a,b,c,d,K}\cap\Omega_{a',b',c',d',K}=\emptyset\right\}$ can be defined, as seen in Proposition \ref{intvacia}, by\[\exists x\exists y\left(abcda'b'c'd'x=1\wedge y\in J_{a,b,c,d,K}\wedge 1-y\in J_{a',b',c',d',K}\right),\]which involves $2+2\cdot 277=556$ quantifiers and can therefore be rewritten as\[\exists x_1\cdots \exists x_{556}\left(\varphi_{193,9}\left(a,b,c,d,a',b',c',d',x_1,\cdots,x_{556}\right)\right),\]where $\varphi_{193,9}\left(a,b,c,d,a',b',c',d',x_1,\cdots,x_{556}\right)$ is the conjunction of $1+2\cdot 96=193$ polynomial equalities of degree at most $\max\left\{9,4\right\}=9$ and having integer coefficients.

Given $a,b,c,d\in K^\times$ and $n\in\mathbb{Z}_{\geq 2}$, the set $J_{a,b,c,d,n,K}$ can be defined as\[\exists x\exists y\left(r=xy^{n-1}\wedge x\in J_{a,b,c,d,K}\wedge y\in J_{a,b,c,d,K}\right),\]which involves $2+2\cdot 277=556$ quantifiers and can therefore be rewritten as\[\exists x_1\cdots\exists x_{556}\left(\varphi_{193,\max\left\{n,4\right\}}\left(a,b,c,d,r,x_1,\cdots,x_{556}\right)\right),\]where $\varphi_{193,\max\left\{n,4\right\}}\left(a,b,c,d,r,x_1,\cdots,x_{556}\right)$ is a conjunction of $1+2\cdot 96=193$ polynomial equalities of degree at most $\max\left\{n,4\right\}$ and having integer coefficients.

Given $a,b,c,d\in K^\times$ and $n\in\mathbb{Z}_{\geq 2}$, the set $\left(J_{a,b,c,d,n,K}\setminus\left\{0\right\}\right)^{-1}$ can be defined as\[\exists y\left(ry=1\wedge y\in J_{a,b,c,d,n,K}\right),\]which involves $1+556=557$ quantifiers. Rewrite this as\[\exists x_1\cdots\exists_{x_{557}}\left(\varphi_{194,\max\left\{n,4\right\}}\left(a,b,c,d,r,x_1,\cdots,x_{557}\right)\right),\]where $\varphi_{194,\max\left\{n,4\right\}}\left(a,b,c,d,r,x_1,\cdots,x_{557}\right)$ is a conjunction of $1+193=194$ polynomial equalities of degree at most $\max\left\{2,\max\left\{n,4\right\}\right\}=\max\left\{n,4\right\}$ and having integer coefficients. Using Proposition \ref{intdiofopt}, we can further rewrite this as\[\exists x_1\cdots\exists x_{557}\left(Q\left(a,b,c,d,r,x_1,\cdots,x_{557}\right)=0\right),\]where $Q$ is a polynomial with integer coefficients having degree at most $194\max\left\{n,4\right\}=\max\left\{194n,776\right\}$. In the case in which $K\subseteq\mathbb{R}$, we can reduce this bound to $2\max\left\{n,4\right\}=\max\left\{2n,8\right\}$, because in such a situation, any formula of the form $a_1=0\wedge \cdots \wedge a_n=0$ can be rewritten as $\displaystyle{\sum_{i=1}^na_i^2=0}$.

From our computations it follows that $\begin{pmatrix}abcda'b'c'd'\neq 0\\ \Omega_{a,b,c,d,K}\cap\Omega_{a',b',c',d',K}=\emptyset\\r\in \left(J_{a',b',c',d',K}\setminus\left\{0\right\}\right)^{-1}\end{pmatrix}$ is an existential formula with $556+278=834$ existential quantifiers which can be rewritten as\[\exists x_1\cdots\exists x_{834}\left(\varphi_{290,9}\left(a,b,c,d,a',b',c',d',r,x_1,\cdots,x_{834}\right)\right),\]where $\varphi_{290,9}\left(a,b,c,d,a',b',c',d',r,x_1,\cdots,x_{834}\right)$ is a conjunction of $193+97=290$ polynomial equalities of degree at most $\max\left\{9,4\right\}=9$. By Proposition \ref{intdiofopt} this can be further rewritten as\[\exists x_1\cdots\exists x_{834}\left(P\left(a,b,c,d,a',b',c',d',r,x_1,\cdots,x_{834}\right)=0\right),\]where $P$ is a polynomial of degree at most $290\cdot 9=2610$ with integer coefficients (or $2\cdot 9=18$ if $K\subseteq\mathbb{R}$).

Finally, by the equivalence \eqref{ae}, the formula\[\forall a'\forall b'\forall c'\forall d'\left[\neg\begin{pmatrix}abcda'b'c'd'\neq 0\\ \Omega_{a,b,c,d,K}\cap\Omega_{a',b',c',d',K}=\emptyset\\r\in \left(J_{a',b',c',d',K}\setminus\left\{0\right\}\right)^{-1}\end{pmatrix}\vee r\in \left(J_{a',b',c',d',n,K}\setminus\left\{0\right\}\right)^{-1}\right]\]involves $4+834=838$ universal quantifiers, $1+557=558$ existential quantifiers, and a polynomial of degree at most $2610+\left(1+\max\left\{194n,776\right\}\right)=\max\left\{194n+2611,3387\right\}$, or $18+\left(1+\max\left\{2n,8\right\}\right)=\max\left\{2n+19,27\right\}$ if $K\subseteq\mathbb{R}$. $\blacksquare$

\end{proof}

It is possible to improve the number of quantifiers by making use of a result by Daans, Dittmann, and Fehm (\cite[Theorem 1.4]{daans2021existential}) by which we can express the conjunction of two existential formulas with $m$ and $n$ existential quantifiers respectively as an existential formula involving $m+n-1$ quantifiers. Revisiting the above computations (ignoring the degree considerations) we get that, for a number field $K$ and $a,b,c,d\in K^\times$, the set $I^c_{a,b,K}$ is existentially defined with $27$ quantifiers, $J_{a,b,K}$ is existentially defined with $107$ quantifiers, $J_{a,b,c,d,K}$ is existentially defined with $214$ quantifiers, and, for a given $n\in\mathbb{Z}_{\geq 2}$, $J_{a,b,c,d,n,K}$ is existentially defined with $429$ quantifiers. Besides, the condition $\begin{pmatrix}abcda'b'c'd'\neq 0\\\Omega_{a,b,c,d,K}\cap\Omega_{a',b',c',d',K}=\emptyset\end{pmatrix}$ is existential with $429$ quantifiers. Putting all this together as above, one obtains the following version of Proposition \ref{volver}, with a smaller number of quantifiers but no control on the degree of the defining polynomial.

\begin{proposition}If $K$ is a number field and $n\in\mathbb{Z}_{\geq 1}$, there exists\[P\in \mathbb{Z}\left[A,B,C,D,R,X_1,\cdots,X_{648},Y_1,\cdots,Y_{431}\right]\]such that, for any $a,b,c,d,r\in K$, the following are equivalent:

\begin{itemize}

\item $abcd=0$, or $abcd\neq 0$ and $\nu_\mathfrak{p}\left(r\right)\in\mathbb{Z}_{\geq 0}\cup\mathbb{Z}_{\leq -n}$ for all $\mathfrak{p}\in\Omega_K^{<\infty}\setminus \Omega_{a,b,c,d,K}$.

\item For all $x_1,\cdots,x_{648}\in K$ there exist $y_1,\cdots,y_{431}\in K$ such that

$P\left(a,b,c,d,r,x_1,\cdots,x_{648},y_1,\cdots,y_{431}\right)=0$.

\end{itemize}

\end{proposition}

\section{Campana points with respect to arbitrary divisors}\label{mainqwe}

By using more general divisors, we get a straightforward generalization to any set of Campana points: if $K$ is a number field and $F\in K\left[x,y\right]$ is an irreducible homogeneous polynomial, define $D\coloneqq \varepsilon \left\{F\left(x,y\right)=0\right\}$ for some $\varepsilon\in\mathfrak{W}$. If $\varepsilon=1$ then the corresponding set of Campana points (with respect to a given finite subset $S$ of $\Omega_K$ containing $\Omega_K^\infty$) induces the set $\left\{\lambda\in K:f\left(\lambda\right)\text{ is integral outside $S$}\right\}$, where $f\left(x\right)\coloneqq F\left(x,1\right)$ is the dehomogenization of $F$ with respect to $y$. Using that $f$ is indeed a polynomial (so we preserve first-order definability), we immediately get universality of these sets from the well-known universal formulas for $S$-integers (\cite{MR3432581}, \cite{MR3207365}, \cite{MR3882159}). But if we consider $\varepsilon=1-\frac{1}{n}$ for some $n\in\mathbb{Z}_{\geq 1}$, we instead get:

\begin{theorem}\label{222}Given a number field $K$, $n\in\mathbb{Z}_{\geq 1}$, and an irreducible homogeneous polynomial \mbox{$F\in K\left[x,y\right]$}, the set\begin{align*}\left\{\left(a,b,c,d,\lambda\right)\in \left(K^\times\right)^4\times K\right.&:abcd=0\\\,&\left.\vee\left[abcd\neq 0\wedge \forall\mathfrak{p}\in\Omega_K^{<\infty}\setminus \Omega_{a,b,c,d,K}\,\left(\nu_\mathfrak{p}\left(F\left(\lambda,1\right)\right)\in\mathbb{Z}_{\geq 0}\cup\mathbb{Z}_{\leq n}\right)\right]\right\}\end{align*}is defined by a $\forall\exists$-formula in $K$ with $838$ universal quantifiers and $558$ existential quantifiers, and a defining polynomial of degree at most $\deg\left(F\right)\max\left\{194n+2611,3387\right\}$ (if $K\subseteq\mathbb{R}$ this degree is at most $\deg\left(F\right)\max\left\{2n+19,27\right\}$) whose parameters are only those involved in $F$.

\end{theorem}

Finally, for any general $\mathbb{Q}$-divisor $\sum_{i=1}^m\varepsilon_i\left\{F_i\left(x,y\right)=0\right\}$ (some $m\in\mathbb{Z}_{\geq 1}$) satisfying the conditions for a Campana Orbifold, the corresponding set of Campana points can be obtained by taking $m$ intersections of such formulas.

\printbibliography

\end{document}